\documentclass[reqno,11pt]{amsart}
\usepackage[text={6.5in,9in},centering]{geometry}

\usepackage{array,multirow}
\usepackage{amsmath}
\usepackage{amssymb}
\usepackage{amsfonts}
\usepackage{graphicx}
\usepackage{bm}
\usepackage{enumerate}
\usepackage{dsfont}
\usepackage{float}

\usepackage[pagewise]{lineno} 
\usepackage{verbatim}
\usepackage{color}

\newtheorem{theorem}{Theorem}[section]
\newtheorem{proposition}[theorem]{Proposition}
\newtheorem{lemma}[theorem]{Lemma}
\newtheorem{corollary}[theorem]{Corollary}

\theoremstyle{definition}
\newtheorem{definition}[theorem]{Definition}

\newtheorem{algorithm}[theorem]{Algorithm}

\theoremstyle{remark}
\newtheorem{remark}[theorem]{Remark}

  \newcounter{mnoter}
  \newcounter{mnoteb}
  \let\oldmarginpar\marginpar
    \renewcommand\marginpar[1]{\-\oldmarginpar[\raggedleft\footnotesize #1]%
    {\raggedright\footnotesize #1}}

\numberwithin{equation}{section}  
\numberwithin{figure}{section}  
\numberwithin{table}{section}  

\renewcommand\lll{|\kern-1pt|\kern-1pt|} 
 
\renewcommand{\O}{\Omega} 

\newcommand{\n}{{\bf n}} 

\newcommand{\jump}[1]{\lbrack\!\lbrack #1 \rbrack\!\rbrack} 
\newcommand{\jtau}{\jump{\taub}} 
\newcommand{\taub}{{\boldsymbol {\tau}}}

\newcommand{\h}{\widetilde{h}}
\newcommand{\Vc}{V^{\rm{conf}}_{\h}}
\newcommand{\Vcr}{V_{h}^{CR}} 

\newcommand{\av}[1]{\{\!\!\{#1\}\!\!\}} 
\newcommand{\Eh}{{{\mathcal E}_h}} 
\newcommand{\Eho}{{{\mathcal E}^{o}_h}} 
\newcommand{\cT}{\mathcal{T}}

\newcommand{\Th}{\mathcal{T}_h}
\newcommand{\K}{T}
\newcommand{\dyle}{\displaystyle}
\newcommand{\Ehb}{{{\mathcal E}^{\partial}_h}} 

\renewcommand{\lor }{\longrightarrow}

\newcommand{\calT}{ \mathcal{T}}

\newcommand{\mathcalB}{\mathcal{B}}
\newcommand{\calA}{ \mathcal{A}}

\newcommand{\calR}{ \mathcal{R}}

\newcommand{\calP}{ \mathcal{P}}

\newcommand{\calE}{ \mathcal{E}}
\newcommand{\calZ}{ \mathcal{Z}}

\newcommand{\triplenorm}[1]{%
 \left\vert\kern-0.9pt\left\vert\kern-0.9pt\left\vert #1
 \right\vert\kern-0.9pt\right\vert\kern-0.9pt\right\vert}  

\newcommand{\Reals}[1]{{\rm I\! R}^{#1}}

\begin{document}

\title[Multilevel Preconditioners for DG methods for  Jump coefficient problems]
      {Multilevel Preconditioners for Discontinuous \\
       Galerkin Approximations of Elliptic Problems \\ 
       with Jump Coefficients}

\author[B. Ayuso de Dios]{Blanca Ayuso De Dios}
\email{bayuso@crm.cat}
\author[M. Holst]{Michael Holst}
\email{mholst@math.ucsd.edu}
\author[Y. Zhu]{Yunrong Zhu}
\email{zhu@math.ucsd.edu}
\author[L. Zikatanov]{Ludmil Zikatanov} 
\email{ltz@math.psu.edu}

\date{\today}


\begin{abstract}
We introduce and analyze two-level and multi-level preconditioners
for a family of Interior Penalty (IP) discontinuous Galerkin (DG) 
discretizations of second order elliptic problems with large jumps in the 
diffusion coefficient. 
Our approach to IPDG-type methods is based on a splitting of the 
DG space into two components that are orthogonal in the energy inner product 
naturally induced by the methods.  
As a result, the methods and their analysis depend in a crucial way on the 
diffusion coefficient of the problem.
The analysis of the proposed preconditioners is presented for both
symmetric and non-symmetric IP schemes; dealing simultaneously with the 
jump in the diffusion coefficient and the non-nested character of the 
relevant discrete spaces presents extra difficulties in the analysis which 
precludes a simple extension of existing results.
However, we are able to establish robustness 
(with respect to the diffusion coefficient) 
and nearly-optimality (up to a logarithmic term depending on the mesh size)
for both two-level and BPX-type preconditioners.
Following the analysis, we present a sequence of detailed numerical results 
which verify the theory and illustrate the performance of the methods.
The paper includes an Appendix with a collection of proofs of 
several technical results required for the analysis.
\end{abstract}

\maketitle



\section{Introduction}
\label{sec:intro}
Let $\Omega\subset \Reals{d}$ be a bounded polygon (for $d=2$) or
polyhedron (for $d=3$) and $f\in L^{2}(\O)$. We consider the following second order 
elliptic equation with strongly discontinuous coefficients:
\begin{equation}\label{eqn:model}
    \left\{\begin{array}{rl}
      -\nabla\cdot(\kappa\nabla u)=f &\mbox{ in } \Omega,\\
      u=0 & \mbox{ on } \partial \Omega.
     \end{array}
    \right.
\end{equation}
The scalar function $\kappa=\kappa(x)$ denotes the diffusion coefficient 
which is assumed to be piecewise constant with respect to an initial 
non-overlapping (open) subdomain partition of the domain $\Omega$,
denoted $\mathcal{T}_S=\{\Omega_m\}_{m=1}^{M},$ with $\cup_{m=1}^M
\overline{\Omega}_m=\overline{\Omega}$ and ${\Omega}_m
\cap {\Omega}_n =\emptyset$ for $n\neq m$. 
Although the (polygonal or polyhedral) regions $\Omega _m\;, m=1\ldots M,$ 
might have complicated geometry, we will always assume that there is an 
initial shape-regular triangulation 
$\mathcal{T}_{0}$ such that $\kappa_{T}=\kappa(x)|_{T}$ is a
constant for all $T\in\mathcal{T}_{0}$.
Problem \eqref{eqn:model} belongs to the class of interface
 or transmission problems, which are relevant to many
applications such as groundwater flow, electromagnetics and semiconductor device modeling. 
The coefficients in these applications might have large discontinuities across the interfaces between different regions with different material properties. 
Finite element discretizations of \eqref{eqn:model} lead to linear
systems with badly conditioned stiffness matrices. The condition numbers of these matrices depend not only on the mesh size, but also on the largest jump in the coefficients. 

Much research has been devoted to developing efficient 
and robust preconditioners  for conforming  finite element discretizations of \eqref{eqn:model}. Nonoverlapping domain decomposition preconditioners, such as Balancing Neumann-Neumann \cite{Mandel.J;Brezina.M1996}, FETI-DP \cite{Klawonn.A;Widlund.O;Dryja.M2002} and Bramble-Pasciak-Schatz
Preconditioners \cite{Bramble.J;Pasciak.J;Schatz.A1989} have been shown to be robust with respect to coefficient variations and mesh size (up to a logarithmic factor), 
in theory and in practice, but only if special {\it exotic}  coarse 
solvers (such as those based on discrete harmonic
extensions \cite{dryja13, Mandel.J;Brezina.M1996, dryja14}) are
used (see also~\cite{Xu.J;Zou.J1998}). 
The construction and use of such exotic coarse spaces is avoided in other multilevel methods, such as the Bramble-Pasciak-Xu (BPX)  or multigrid preconditioners, for which it has always been observed that when used with conjugate gradient (CG) iteration, 
result in robust and  efficient algorithms with respect 
to jumps in the coefficients, independently of the problem dimension.
However, their analysis (based on the standard CG theory) predict a deterioration in the rate of convergence with respect to both the coefficients and the mesh size, 
By resorting to more sophisticated CG theory (see \cite[Section 13.2]{Axelsson.O1994}, \cite{Axelsson.O2003}) which accounts for and exploits the particular spectral structure of the  preconditioned systems\footnote{Namely, that there are a few small eigenvalues due to the jump coefficient  distribution that have no influence in the (observed) overall convergence of the iteration}, the authors in \cite{XuJZhuY-2008aa,ZhuY-2008aa} show that standard multilevel and overlapping domain decomposition methods lead to nearly optimal preconditioners for CG algorithms. (See also \cite{Chen.L;Holst.M;Xu.J;Zhu.Y2010}). Much less attention has been devoted to nonconforming approximations. Overlapping preconditioners for the lowest order Crouzeix-Raviart approximation of \eqref{eqn:model} are found in \cite{sarkisNC0,sarkisNC1}, where the analysis depends on the assumption that the coefficient $\kappa$ is quasi-monotone.

In this article, we consider the construction and analysis of preconditioners for the 
Interior Penalty (IP) Discontinuous Galerkin (DG) approximation of \eqref{eqn:model}.
Based on discontinuous finite element spaces, DG methods can deal robustly with partial differential equations
of almost any kind, as well as with equations whose type changes within the computational
domain. They are naturally suited for multi-physics applications, 
and for problems with highly varying material properties, 
such as \eqref{eqn:model}.
The design of efficient solvers for DG discretizations has been 
pursued only in the last ten years;  and, while classical approaches have been successfully extended 
to second order elliptic problems, the discontinuous nature of
the underlying finite element spaces has motivated the creation of
new techniques to develop solvers. Additive Schwarz methods (of overlapping and non-overlapping type)
are considered and analyzed in \cite{kara0, sarkis0,
   AntoniettiPAyusoB-2007aa, AntoniettiPAyusoB-2008aa,
   AntoniettiPAyusoB-2010aa, bren10}. 
Multigrid methods are studied 
in \cite{GopalakrishnanJKanschatG-2003aa, brennerVF04, wonipgW,
   bren11, SA_DG0, CockbDuboiGopal10}. 
Two-level methods are presented
in~\cite{DobrevVLazarovRVassilevskiPZikatanovL-2006aa,
   BrixKCampos-PintoMDahmenW-2008aa,
   BrixKCampos-PintoMDahmenWMassjungR-2009aa}.
More general multi-level methods based on algebraic techniques are considered
in~\cite{kraus0,kraus1}. However, all the analysis in these works consider only the case of a smoothly or slowly varying diffusivity coefficient.
For problem \eqref{eqn:model}, only in \cite{sarkis0, sarkis1, sarkis2} 
have the authors introduced and analyzed non-overlapping BBDC and
FETI-DP domain decomposition preconditioners for a Nitsche type method where a Symmetric Interior Penalty  DG discretization is used (only) on the skeleton of the subdomain partition, 
while a standard conforming approximation is used in the interior of the
subdomains. 
Robustness and quasi-optimality is shown in $d=2$ for the Additive
and Hybrid BBDC \cite{sarkis1} and FETI-DP \cite{sarkis2}
preconditioners, even for the case of non-matching grids. As it happens for conforming discretizations, the construction and analysis of these preconditioners rely on the use of  {\it exotic} coarse solvers, which might complicate the actual implementation of the method.

The goal of this article is to design, and provide a rigorous analysis of,
a {\it simple} multilevel solver for the lowest order (i.e. piecewise linear discontinuous) approximation of a family of Interior Penalty (IPDG) methods.  To ease the presentation, we focus on a minor variant of the classical IP methods, penalizing only the mean value of the jumps: the ``weakly penalized'' or IPDG-0 methods (called Type-0 in \cite{Ayuso-de-DiosBZikatanovL-2009aa}). 
Our approach follows the ideas in \cite{Ayuso-de-DiosBZikatanovL-2009aa}, and it is based on a splitting of the DG space into two components that are orthogonal in the energy inner product naturally induced  by the IPDG-0 methods.
 
Roughly speaking, the construction amounts to
identifying a ``low frequency'' space (the Crouzeix-Raviart elements) and
then defining a  complementary space. However, a notable difference takes place in the DG space decomposition introduced for the Laplace equation~\cite{BurmanEStammB-2008aa,
  Ayuso-de-DiosBZikatanovL-2009aa}. For problem \eqref{eqn:model}, the subspaces depend on the coefficient $\kappa$, and 
this is certainly related to the splittings used in algebraic multigrid
(AMG~\cite{BrandtMcCormickRuge_1982aa}).
With the orthogonal splitting of the DG space at hand, the solution of problem
\eqref{eqn:model} reduces to solving two sub-problems: a non-conforming
approximation to \eqref{eqn:model}, and a problem in the complementary
space containing {\it high oscillatory} error components.
We show the latter subproblem is easy to solve, since it is spectrally
equivalent to its diagonal form, and so  CG with a diagonal
preconditioner is a uniform and robust solver.

For the former subproblem, following \cite{XuJZhuY-2008aa,ZhuY-2008aa}, we develop and analyze (in the standard and asymptotic convergence regimes) a two-level method and a BPX preconditioner. Nevertheless, dealing simultaneously with the jump in the coefficient 
$\kappa$ and the non-nested character of the Crouziex-Raviart (CR) spaces presents extra
difficulties in the analysis which precludes a simple extension
of \cite{XuJZhuY-2008aa,ZhuY-2008aa}.
We are able to establish nearly optimal convergence and robustness 
(with respect to both the mesh size and the coefficient $\kappa$) 
for the two-level method and for the BPX preconditioner 
(up to a logarithmic term depending on the mesh size). The resulting algorithms involve the use of a solver in the CR space 
that is reduced to a smoothing step followed by a conforming solver.
Therefore, in particular one can argue that any of the robust and 
efficient solvers designed for conforming approximations of 
problem \eqref{eqn:model} could be used as a preconditioner here.  Finally we mention that, although the two-level and multilevel methods we propose are based on the piecewise linear IP-0 methods, they could be used as preconditioners for the solution of the linear systems arising from high order DG methods. 
 
\subsection*{Outline of the paper}

The rest of the paper is organized as follows. We introduce the IPDG-1
and IPDG-0 methods for approximating \eqref{eqn:model} in~\S\ref{sec:dg} and revise some of their properties. The space
decomposition of DG finite element space is introduced in~\S\ref{sec:Vdecomp}. Consequences of the space splitting are described in~\S\ref{sec:solvers}. The
two-level and multi-level methods for the Crouzeix-Raviart
approximation of  \eqref{eqn:model} are constructed and analyzed in~\S\ref{sec:crsolver}.  Numerical experiments are included
in~\S\ref{sec:numerical}, to verify the theory and assess the
performance and robustness of the proposed preconditioners. 
In~\S\ref{subsect:preconditioners-for-type-0 and-1} we briefly comment on how the developed solvers and theory can be extended for the classical IPDG-1 family. The paper
is completed with an Appendix where we have collected proofs of several technical results required in our analysis.


Throughout the paper we shall use the
standard notation for Sobolev spaces and their norms.  
We will use the notation $x_1\lesssim y_1$, and $x_2\gtrsim y_2$,
whenever there exist constants $C_1, C_2$ independent of the mesh size
$h$ and the coefficient $\kappa$ or other parameters that $x_1$,
$x_2$, $y_1$ and $y_2$ may depend on, and such that $x_1 \le C_1 y_1$
and $x_2\ge C_2 y_2$, respectively. We also use the notation $x\simeq y$ for $C_{1} x \le y\le C_{2} x$.
\section{Discontinuous Galerkin Methods}
\label{sec:dg}

In this section, we introduce the basic notation and describe the DG
methods we consider for approximating the problem
\eqref{eqn:model}. 

Let $\Th$ be a shape-regular family of partitions of $\O$ into
$d$-simplices $\K$ (triangles in $d=2$ or tetrahedra in $d=3$). We
denote by $h_{\K}$ the diameter of $\K$ and we set $h=\max_{\K \in
  \Th} h_{\K}$. We also assume that the decomposition $\Th$ is
conforming in the sense that it does not contain hanging nodes and
that $\Th\subset \mathcal{T}_0$, with $\mathcal{T}_0$ being quasi-uniform initial triangulation that resolves the coefficient $\kappa$.  We denote
by $\Eh$ the set of all edges/faces and by $\Eho$ and $\Ehb$ the
collection of all interior and boundary edges/faces, respectively.
The space $H^{1}(\Th)$ is the set of element-wise $H^{1}$ functions,
and $L^{2}(\Eh)$ refers to the set of functions whose traces on
the elements of $\Eh$ are square integrable.

Following \cite{ArnoldDBrezziFCockburnBMariniL-2001aa}, we recall the
usual DG analysis tools. Let $\K^{+}$ and $\K^{-}$ be
two neighboring elements, and let $\n^{+}$, $\n^{-}$ be their outward
normal unit vectors, respectively ($\n^{\pm}=\n_{\K^{\pm}}$). Let
$\zeta^{\pm}$ and $\taub^{\pm}$ be the restriction of $\zeta$ and
$\taub$ to $\K^{\pm}$. We set:
\begin{align*}
2\av{\zeta}&=(\zeta^+ +\zeta^-),\quad
\jump{\zeta}=\zeta^+\n^++\zeta^-\n^- \quad&&\mbox{on } e\in
\Eho,\label{av-jump-scalar}\\
2\av{\taub}&=(\taub^+ +\taub^-), \quad
\jtau=\taub^+\cdot\n^+ + \taub^-\cdot\n^- &&\mbox{on } e\in
\Eho. \nonumber 
\end{align*}
We also define the weighted average, $\av{\cdot}_{\delta}$, for any $\delta=\{\delta_e\}_{e\in\Eho}$ with $\delta_e \in [0,1]\,\, \forall\,e$:
\begin{equation}\label{av-w}
\av{\zeta}_{\delta_{e}}=\delta_e \zeta^+ + (1-\delta_e)\zeta^-\;,\quad \av{\taub}_{\delta_{e}}=\delta_e\taub^+ +(1-\delta_e)\taub^-\;, \qquad \mbox{on } e\in
\Eho\;.
\end{equation}
For $e\in \Ehb$, we set
\begin{equation}\label{av-jump-boundary}
\jump{\zeta}=\zeta\n, \quad
\av{\taub}=\av{\taub}_{\delta_{e}}= \taub \qquad \mbox{on } e\in \Ehb.
\end{equation}
We will also use the notation
\begin{equation*}
(u,w)_{\Th}=\dyle\sum_{\K \in \Th} \int_{\K} uw dx \quad \forall\,\, u,w \in L^{2}(\Omega),\quad \langle u, w\rangle_{\Eh}=\dyle\sum_{e\in \Eh} \int_{e} u w ds \quad \forall\, u,w, \in L^{2}(\Eh).
\end{equation*}
The DG approximation to the model problem \eqref{eqn:model} can be written as
\begin{equation*}\label{25}
\mbox{ Find     }u^{DG}_h \in V_{h}^{DG} \mbox{   such that   }\calA^{DG}(u^{DG}_{h},w)=(f,w)_{\Th}\;, \quad \forall\, w\in V_{h}^{DG}\;,
\end{equation*}
where $V_{h}^{DG}$ is the piecewise linear discontinuous finite element space, and $\calA^{DG}(\cdot,\cdot)$ is the bilinear form defining the method. 

In this paper, we focus on a family of weighted Interior Penalty methods (see \cite{StenbergR-1998aa}), with special attention given to a variant (weakly penalized) of them. The bilinear form defining the {\it classical} family of  weighted IP methods~\cite{StenbergR-1998aa}, here called IP($\beta$)-1 methods,  is given by $\calA^{DG}(\cdot,\cdot)=\calA(\cdot,\cdot)$, with 
\begin{equation}\label{ipA}
\begin{aligned}
\calA(v,w)&= (\kappa\nabla _h v,\nabla w)_{\Th} -\langle \av{ \kappa\nabla v }_{\beta_{e}}, \jump{w}\rangle_{\Eh}+\theta \langle \jump{v},\av{\kappa\nabla w}_{\beta_{e}} \rangle_{\Eh} &&\\
&\qquad +\langle \alpha h^{-1}_{e}  \kappa_e \jump{v},\jump{w}\rangle_{\Eh},   \qquad \forall\, v,\, w\,\, \in \, V_{h}^{DG}\;.&&
\end{aligned}
\end{equation}
where $\theta =-1$ gives the SIPG($\beta$)-1 methods; $\theta = 1$
leads to NIPG($\beta$)-1 methods; and $\theta =0$ gives the IIPG($\beta$)-1
methods. Here, $h_e$ denotes the $(d-1)$ dimensional Lebesgue measure of $e\in \Eh$.The penalty parameter $\alpha>0$ is set to be a
positive constant; and  it has to be taken large enough
to ensure coercivity of the corresponding bilinear forms when $\theta\ne 1$. 
The symmetric method was first considered in  \cite{StenbergR-1998aa} and later 
in \cite[Section 4]{Dryja2003} for jump coefficient problems 
(although there it was written using a slightly different
notation and DG was only used in the skeleton of the partition). It was later extended to advection-diffusion problems in
\cite{BurmanEZuninoP-2006aa} and
\cite{Di-PietroDErnAGuermondJ-2008aa}.

We also introduce the corresponding family of IP($\beta$)-0 methods,
which use the mid-point quadrature rule for computing the integrals in
the last term in \eqref{ipA} above. That is, we set
$\calA^{DG}(\cdot,\cdot)=\calA_0(\cdot,\cdot)$ with
\begin{equation}\label{ipA0}
\begin{aligned}
\calA_0(v,w)&= (\kappa\nabla v,\nabla w)_{\Th} -\langle \av{ \kappa\nabla v}_{\beta_{e}}, \jump{w}\rangle_{\Eh}+\theta \langle \jump{v},\av{\kappa\nabla w}_{\beta_{e}} \rangle_{\Eh} &&\\
&\qquad +\langle \alpha h^{-1}_{e}  \kappa_{e} \calP^{0}_e(\jump{v}),\jump{w}\rangle_{\Eh},   \qquad \forall\, v,\,\, w\,\, \in V_{h}^{DG}\;,&&
\end{aligned}
\end{equation}
where $\calP^{0}_e: L^{2}(\Eh)\mapsto \mathbb{P}^{0}(\Eh)$ is the
$L^{2}$-projection onto the piecewise constants on $\Eh$. 
We note that this projection satisfies $\|\calP^{0}_e\|_{L^{2}(\Eh)}=1$. 
In \eqref{ipA} and \eqref{ipA0}, for any $e\in \Eho$ with $e=\partial
T^{+}\cap \partial T^{-}$,  the coefficient $\kappa_{\K}$ and the 
{\it weight} $\beta_e$ are defined as follows: 
\begin{equation}\label{def:beta}
\kappa_{\K}=\kappa|_{\K}, \quad
\beta_e=\frac{\kappa^{-}}{\kappa^{+}+\kappa^{-}},
\quad\mbox{where}\quad \kappa^{\pm}=\kappa_{|_{\K^{\pm}}},
\end{equation}
The coefficient $\kappa_e$ as the harmonic mean of
$\kappa^{+}$ and $\kappa^{-}$:
\begin{equation}\label{defKE}
\kappa_e:=\frac{2\kappa^{+}\kappa^{-}}{\kappa^{+}+\kappa^{-}}\;.
\end{equation} 
The weight $\beta=\{\beta_e\}_{e\in \Eho}$ depends on the 
coefficient $\kappa$ and therefore it might vary over all interior edges/faces (of the subdomain partition $\mathcal{T}_0$ resolving the coefficient $\kappa$). 

\begin{remark}
We note that one could take $\kappa_{e}$ as $\min\{\kappa^{+}, \kappa^{-}\}$,
since both are equivalent:
\begin{equation}\label{cota-a}
\min{\{\kappa^{+},\kappa^{-}\}} \leq  \kappa_{e} =\frac{2\kappa^{+}\kappa^{-}}{\kappa^{+}+\kappa^{-}} \le
2\min{\{\kappa^{+},\kappa^{-}\}} \le 2\kappa^{\pm}\;.
\end{equation}
The equivalence relations in \eqref{cota-a} show that the results on
spectral equivalence and uniform preconditioning given later for
\eqref{ipA} with $\kappa_{e}$ defined in~\eqref{defKE} (the
harmonic mean) will automatically hold for method \eqref{ipA}
with $\kappa_e:=\min{\{\kappa^{+},\kappa^{-}\}}$. To fix the
notation and simplify the presentation, we stick to definition \eqref{defKE} for 
$\kappa_{e}$.
\end{remark}
\subsection*{Weighted Residual Formulation}
Following \cite{BrezziFCockburnBMariniLSuliE-2006aa}
 we can rewrite the two
families of IP methods in the weighted residual framework: For all
$v, w\in V_{h}^{DG}$,
\begin{align}
\calA(v,w)&=(-\nabla\cdot (\kappa\nabla v), w)_{\Th} +\langle \jump{\kappa\nabla  v}, \av{w}_{1-\beta_{e}}\rangle_{\Eho} +\langle \jump{v}, \mathcalB_{1}(w)\rangle_{\Eh}, && \label{ip:res}\\
\calA_0(v,w)&=(-\nabla\cdot (\kappa\nabla v), w)_{\Th} +\langle \jump{\kappa\nabla  v}, \av{w}_{1-\beta_{e}}\rangle_{\Eho} +\langle \jump{v}, \calP^0_{e}(\mathcalB_{1}(w))\rangle_{\Eh}, && \label{ip:res0}
\end{align}
where $\mathcalB_1$ is defined as:
\begin{equation}\label{defB1}
\mathcalB_{1}(w) = \theta \av{\kappa\nabla w}_{\beta_{e}}
+\alpha h_e^{-1}\kappa_{e}\jump{w}, \qquad \forall\, e\in \Eh.
\end{equation}
 Throughout the paper both the weighted residual formulation
\eqref{ip:res}-\eqref{ip:res0} and the standard one
\eqref{ipA}-\eqref{ipA0} will be used interchangeably.

We now establish a result that guarantees the spectral equivalence
between $\mathcal{A}(\cdot,\cdot)$ and $\mathcal{A}_0(\cdot,\cdot)$.
\begin{lemma}
	\label{lm:equivA:A0}
	Let $\mathcal{A}(\cdot,\cdot)$ be a bilinear form
        corresponding to a  IP($\beta$)-1 method \eqref{ipA} and
        let $\mathcal{A}_{0}(\cdot,\cdot)$ be the corresponding IP($\beta$)-0 
        bilinear form as defined in \eqref{ipA0}.  Then there exists a
        positive constant $c_{0}=c_{0}(\alpha)$, depending only on the
        shape regularity of the mesh and the penalty parameter
        $\alpha$ (but independent of the coefficient $\kappa$ and the mesh
        size $h$) such that,
\begin{equation}\label{equivA:A0}
\calA_{0} (v,v)\le \calA(v,v)\le c_0(\alpha)  \calA_{0} (v,v)\quad \forall v\in V_{h}^{DG}.
\end{equation}
\end{lemma}
\begin{proof}
  The lower bound follows immediately from the fact that the
  projection $\calP_{e}^{0}$ is an $L^{2}(\Eh)$-orthogonal projection and therefore
  has unit norm.  The upper bound would follow if we  show 
  \begin{equation*}
    \dyle{\sum_{e\in\Eh} \alpha h_e^{-1}\kappa_{e}
      \|\jump{v}\|_{0,e}^{2} \le C (\sum_{\K\in \Th}
      \kappa_{\K}\|\nabla v\|_{0,\K}^{2} +\sum_{e\in\Eh}  \alpha h_e^{-1}\kappa_{e} \| \calP^{0}_{e}\jump{v}\|_{0,e}^{2}})\;,
\end{equation*}
which can be proved by arguing exactly as in \cite{Ayuso-de-DiosBZikatanovL-2009aa, BrennerSOwensL-2007aa, abm} and taking into account~\eqref{cota-a}.
 \end{proof}
By virtue of Lemma \ref{lm:equivA:A0}, it will be enough throughout the rest of the paper to focus on the design and analysis of multilevel preconditioners for the IP($\beta$)-0 methods. At least in the symmetric case, the  preconditioners proposed for SIPG($\beta$)-0 will exhibit the same convergence (asymptotically) when applied to SIPG($\beta$)-1.

\subsection*{Continuity and Coercivity of IP($\beta$)-0 methods}
The family of methods \eqref{ipA0} can be shown to provide an
accurate and robust approximation to the solution of
\eqref{eqn:model}.  
We define the energy norm:
\begin{equation}\label{def-normN}
 \triplenorm{v}_{DG0}^{2}:= \dyle\sum_{T\in \Th} \kappa_{T}\|\nabla v\|_{0,T}^{2}+ \dyle\sum_{e\in \calE_{h}}\kappa_{e}h^{-1}_{e}\|\calP^{0}_{e}(\jump{v})\|_{0,e}^{2}.
\end{equation}
Then,  $\calA_0(\cdot,\cdot)$
is \emph{continuous} and \emph{coercive} in the above norm, with
constants independent of the mesh size $h$ and the coefficient
$\kappa$:
\begin{align}
&\mbox{\it Continuity:}& |\calA_0(v,w)| & \lesssim \triplenorm{v}_{DG0} \,\triplenorm{ w}_{DG0},
\qquad &\forall\, v\;, \,\, w\,\in \,V_{h}^{DG},&&
&& \label{eq:cont}
\\
&\mbox{\it Coercivity:} &\calA_0(v,v) & \gtrsim \triplenorm{v}_{DG0}^{2}\;,\qquad\qquad &\forall v\in V_{h}^{DG0}\;.&&
\label{eq:coer}
\end{align}
Although the proof of~\eqref{eq:coer} and~\eqref{eq:cont} is standard,  
we sketch it here for completeness. 
Note first that for each $e\in \Eho$ such that
$e=\partial\K^{+}\cap \partial\K^{-}$, the weighted average $ \av{
 \kappa\nabla  v}_{\beta_e}$ can be rewritten as:
\begin{align}
\av{\kappa \nabla v}_{\beta_e}&=\beta_e (\kappa^{+}(\nabla v)^{+})+(1-\beta_{e})
(\kappa^{-}(\nabla v)^{-})&&\nonumber\\
& =\frac{\kappa^{-}}{\kappa^{+}+\kappa^{-}} \kappa^{+} (\nabla v)^{+} + \frac{\kappa^{+}}{\kappa^{+}+\kappa^{-}} \kappa^{-} (\nabla v)^{-} &&\nonumber\\
&= \frac{\kappa^{+}\kappa^{-}}{\kappa^{+}+\kappa^{-}} [ (\nabla v)^{+} + (\nabla v)^{-}] =\kappa_{e}\av{\nabla v}\;.&& 
 \end{align} 
Trace inequality~\cite{AgmonS-1965aa},
inverse inequality~\cite{CiarletP-1978aa} and~\eqref{cota-a} 
 imply the following bounds
\begin{equation*}
\begin{aligned}
h_e\|\av{\kappa \nabla v}_{\beta_{e}}\|^{2}_{0,e} &\le C_t(\kappa_{e})^{2}
\left(\|\nabla v\|^{2}_{0,\K^{+}\cup \K^{-}} +h^{2}|\nabla v|^{2}_{1,\K^{+}\cup \K^{-}}
\right)\\ &\le 2(\kappa_{e}) C_t(1+C_{inv}^{2}) \left(\kappa^{+}\|\nabla
v\|^{2}_{0,\K^{+}}  + \kappa^{-}\|\nabla
v\|^{2}_{0,\K^{-}} \right).
\end{aligned}
\end{equation*}
This inequality, combined with Cauchy-Schwarz inequality and \eqref{cota-a},
gives 
\begin{equation}\label{cota:mix}
\begin{aligned}
  \left|\langle \av{\kappa \nabla v}_{\beta_e}, \jump{w}\rangle_{\Eh}\right|& =\left|\sum_{e\in\Eh}\int_{e} \kappa_{e}\av{\nabla v} \calP^{0}_e(\jump{w}) ds\right| && \nonumber \\
  &\le \left(\sum_{e\in\Eh} \frac{1}{\alpha}h_{e}\kappa_{e} \|\av{\nabla v}\|^{2}_{0,e} \right)^{1/2} \left(\sum_{e\in\Eh} \alpha h_{e}^{-1} \kappa_{e} \|\calP^{0}_e(\jump{w})\|_{0,e}^{2}\right)^{1/2} &&\nonumber\\
  &\le \frac{8C_t(1+C^{2}_{inv})}{\alpha}\sum_{T\in\Th}
  \kappa_{\K}\|\nabla v\|^{2}_{0,\K} + \frac{1}{4}\sum_{e\in \Eh} \alpha
  h_e^{-1}\kappa_{e}\| \calP^{0}_e(\jump{w})\|_{0,e}^{2}.
\end{aligned}
\end{equation}
Now \eqref{eq:cont} follows from Cauchy-Schwarz inequality. 
The inequality \eqref{eq:coer} is proved by setting $w=v$ in
\eqref{ipA} and taking into account the above estimate. We have then
\begin{align*}
  \calA_0(v,v) &= \sum_{\K\in\Th} \kappa_{\K}\|\nabla v\|_{0,\K}^{2} + \alpha\sum_{e\in\Eh} \kappa_{e}h_e^{-1}\|\calP^{0}_e(\jump{v})\|_{0,e}^{2} -(1-\theta)\langle \av{\kappa \nabla v}_{\beta_e}, \jump{v}\rangle_{\Eh} &&\\
  &\ge \triplenorm{v}_{DG}^{2} -|1-\theta| \left|\langle \av{\kappa \nabla u}_{\beta_e}, \calP^{0}_e(\jump{v})\rangle_{\Eh}\right| &&\\
  &\ge \left(1-
    \frac{8C_t(1+C^{2}_{inv})}{\alpha}\right)\sum_{\K\in\Th}
  \kappa_{\K}\|\nabla v\|_{0,\K}^{2} +
  \frac{4-|1-\theta|}{4}\alpha\sum_{e\in\Eh}
  \kappa_{e}h_e^{-1}\| \calP^{0}_e(\jump{v})\|_{0,e}^{2}\;, &&
\end{align*}
and \eqref{eq:coer} follows immediately by taking $\alpha\geq 1$ large
enough (if $\theta\ne 1$). Moreover, notice that both constants in
\eqref{eq:cont} and \eqref{eq:coer} depend on the shape regularity of
the mesh partition but are independent of the coefficient $\kappa$. \\

Obviously, continuity and coercivity also hold for the IP($\beta$)-1 methods~\eqref{ipA} if the norm \eqref{def-normN}  is replaced by
\begin{equation}\label{def-norm1}
 \triplenorm{v}_{DG}^{2}:= \dyle\sum_{T\in \Th} \kappa_{T}\|\nabla v\|_{0,T}^{2}+ \dyle\sum_{e\in \calE_{h}}\kappa_{e}h^{-1}_{e}\|\jump{v}\|_{0,e}^{2}.
\end{equation}
See \cite{Dryja2003} or \cite{ahzz:tech} for a detailed proof. For both families of methods, optimal error estimates in the energy norms \eqref{def-normN} and  \eqref{def-norm1}  can be shown, arguing as in~\cite{ArnoldDBrezziFCockburnBMariniL-2001aa}. See also \cite{abm} for further discussion on the $L^{2}$-error analysis of these methods.

\section{Space decomposition of the $V_{h}^{DG}$ space}
\label{sec:Vdecomp}

In this section, we introduce a decomposition of the $V_{h}^{DG}$-space
that will play a key role in the design of the solvers for the DG
discretizations \eqref{ipA} and \eqref{ipA0}. In
\cite{Ayuso-de-DiosBZikatanovL-2009aa, BurmanEStammB-2008aa}, it is shown that the
discontinuous piecewise linear finite element space $ V_{h}^{DG}$ admits
the decomposition: $V_{h}^{DG} = V_{h}^{CR} \oplus \calZ$, where
$V_{h}^{CR}$ denotes the standard Crouzeix-Raviart space defined as
\begin{equation}\label{defCR}
V_{h}^{CR}=\left\{ v\in L^{2}(\Omega) \, : \, v_{|_{\K}}\,
\in\, \mathbb{P}^{1}(\K) \,\, \forall \K \in \Th\, \mbox{   and    }
\calP_{e}^{0} 
(\jump{v}\cdot \n)=0 \,\, \forall\, e\in \Eho\right\},
\end{equation}
and the complementary space $\calZ$ is a space of piece-wise linear functions
with average zero at the mass centers of the internal edges/faces: 
\begin{equation*}
\calZ=\left\{ z\in L^{2}(\Omega) \,\, : \,\, z_{|_{\K}}\,\,
\in\,\, \mathbb{P}^{1}(\K) \,\, \forall \K \in \Th\, \mbox{    and    }
\calP_{e}^{0} 
(\av{v})=0, \,\, \forall\, e\in \Eho \right\}.
\end{equation*}
In~\cite{Ayuso-de-DiosBZikatanovL-2009aa}, it was shown that this
decomposition satisfies $\calA_0(v,z)=0$ when $\kappa\equiv 1$, for all $v\in V_{h}^{CR}$
and $z\in \calZ$.
We now modify the definition of $\calZ$ above in
order to account for the presence of a coefficient in the
problem~\eqref{eqn:model}. Let
\begin{equation}\label{defZ}
\calZ_{\beta}=\left\{ z\in L^{2}(\Omega) \,\, : \,\, z_{|_{\K}}\,\,
\in\,\, \mathbb{P}^{1}(\K) \,\, \forall \K \in \Th\, \mbox{    and    }
\calP_{e}^{0} (
\av{z}_{1-\beta_{e}})=0, \,\, \forall\, e\in \Eho \right\},
\end{equation}
where the weight $\beta_{e}$ was defined earlier in~\eqref{def:beta}. Note 
that the weight $\beta_{e}$ depends on the coefficient $\kappa$,
and, as a consequence, 
the space $\calZ_{\beta}$ is also coefficient dependent.
In what follows, we shall show that $\calZ_{\beta}$ is a space complementary to
$V_{h}^{CR}$ in $V_{h}^{DG}$ and the corresponding decomposition has
properties analogous to the properties of the decomposition 
$V_{h}^{DG}=V_{h}^{CR}\oplus\calZ$ given
in~\cite{Ayuso-de-DiosBZikatanovL-2009aa}  for the Poisson problem.
 
For any $e\in \Eh$ with $e\subset T\in \Th$, let
$\varphi_{e,T}$ be the canonical Crouzeix-Raviart basis function on
$T$, which is defined by
$$
	\varphi_{e,T}|_{T} \in \mathbb{P}^{1}(T), \quad 
	\varphi_{e,T}(m_{e'}) = \delta_{e,e'}\;\; \forall e'\in \Eh(T),
	\mbox{  and  } \varphi_{e,T} (x) =0\;\; \forall x\not\in T,
$$
where $m_{e}$ is the mass center of $e$.
We will denote by $n_{T}$ and $n_{E}$ 
the number of simplices and faces (or edges when $d=2$)
respectively. 
We also denote by $n_{BE}$ the number of boundary faces.
\begin{proposition}\label{teo0}
For any $u\in V_{h}^{DG}$ there exists a unique $v\in V_{h}^{CR}$ and a unique $z_{\beta}\in \calZ_{\beta}$ such that $u=v+z_{\beta}$ , that is
\begin{equation}\label{splitting0}
V_{h}^{DG}=V_{h}^{CR}\oplus \calZ_{\beta} .
\end{equation}
\end{proposition}
\begin{proof}
  For simplicity, throughout the proof we will set
  $\beta^{+}=\beta_{e}$, $\beta^{-}=(1-\beta_{e})$, and
  $\varphi^{\pm}_{e} = \varphi_{e,T^{\pm}}$ for any $e\in \Eho$ with
  $e=\partial T^{+}\cap \partial T^{-}.$ We also denote
  $\varphi_{e}=\varphi_{e,T}$ for any $e\in \Ehb$ with $e=\partial
  T\cap \partial \Omega$. Since the mesh is made of $d$-simplices
$$
\dim V_{h}^{DG} = (d+1)n_{T} = 2n_{E} -n_{BE},
$$ 
and it is also obvious that $\{\varphi_{e}^{\pm}\}_{e\in \Eho}\cup
\{\varphi_{e}\}_{e\in \Ehb}$ form a basis for $V_{h}^{DG}.$ Notice that
$\beta^{+}+\beta^{-}=1$, we can therefore express any $u\in V_{h}^{DG}$ as
\begin{eqnarray*}
 u(x)&=&\sum_{e\in \Eho} u^{+}(m_e)\varphi_{e}^{+}(x)+\sum_{e\in \Eho}
 u^{-}(m_e)\varphi_{e}^{-}(x)+\sum_{e\in \Ehb} u(m_e)\varphi_{e}(x)\\
& =& \sum_{e\in \Eho} (\beta^{-}u^{+}(m_e)+\beta^{+}u^{-}(m_e))(\varphi_{e}^{+}(x)+\varphi_{e}^{-}(x))\\
 &&~~~~+\sum_{e\in \Eho} (u^{+}(m_e)-u^{-}(m_e))(\beta^{+}\varphi_{e}^{+}(x)-\beta^{-}\varphi_{e}^{-}(x)) + \sum_{e\in \Ehb} u(m_e)\varphi_{e}(x)\\
 &=& \sum_{e\in \Eho} \left(\frac{1}{|e|}\int_{e}\av{u}_{1-\beta_{e}} ds
 \right) (\varphi_{e}^{+}(x)+\varphi_{e}^{-}(x)) \\
&& ~~~~+\sum_{e\in \Eho}
\left(\frac{1}{|e|}\int_{e}\jump{u}\n^{+}ds
\right)(\beta^{+}\varphi_{e}^{+}(x)-\beta^{-}\varphi_{e}^{-}(x))  +\sum_{e\in \Ehb}  \left(\frac{1}{|e|}\int_{e}\jump{u}\n ds \right) \varphi_{e}(x)\\
 & =& v(x)+z_{\beta}(x). 
 \end{eqnarray*}
Then for each $e\in \Eho$,  we set
\begin{equation}
\label{basis-CR}
\varphi_{e}^{CR}(x) := \varphi_{e}^{+} (x) + \varphi_{e}^{-}(x),
\end{equation}
 \begin{equation}\label{basis-z}
 \psi_{e}^{z}(x) := \beta^{+}\varphi_{e}^{+}(x)-\beta^{-}\varphi_{e}^{-}(x) =\left\{
\begin{array}{rl}
\beta^{+}\varphi_{e}^{+}(x), &\quad x\in \K^{+}\\
- \beta^{-}\varphi_{e}^{-}(x), &\quad x\in \K^{-}
\end{array}\right. ,
 \end{equation}
 and $\psi_{e}^{z}(x) := 0$ for all $x\not\in T^{+}\cup T^{-}.$
 In the definition \eqref{basis-z} of $\psi_{e}^{z}(x)$, we have used 
 $\varphi_{e}^{-}(x)=0$ for $x\in \K^{+}$ and
 $\varphi_{e}^{+}(x)=0$ for $x\in \K^{-}.$
Finally, when $e\in \Ehb$ with $e = \partial\K\cap \partial\Omega$
for some $\K$,  we set
\begin{equation}\label{basis-zb}
\psi_{e}^{z}(x)=\varphi_{e}(x),\quad \forall x\in T.
\end{equation}
It is then straightforward to check that  
\begin{equation*}
V_{h}^{CR}=\operatorname{span}\{\varphi_{e}^{CR}\}_{e\in  \Eho}, \quad  \mbox{and}\quad
\calZ_{\beta}=\operatorname{span}\{\psi^{z}_{e}\}_{e\in \Eh}.
\end{equation*}
Hence, for all $u\in V_{h}^{DG}$ there exist unique $v\in V_{h}^{CR}$ and $z_{\beta}\in \calZ_{\beta}$ defined by 
\begin{eqnarray*}
v&=&\sum_{e\in \Eho}
\left(\frac{1}{|e|}\int_{e}\av{u}_{1-\beta_{e}}ds
\right)\varphi^{CR}_{e}(x) \in V_{h}^{CR}, \\
z_{\beta}&=&\sum_{e\in \Eh}  
\left(\frac{1}{|e|}\int_{e}\jump{u}\n^{+}ds \right) \psi^{z}_{e}(x) \in \calZ_{\beta},
\end{eqnarray*}
such that $u=v+z_{\beta}$. This shows~\eqref{splitting0}
and concludes the proof.
\end{proof}

\begin{remark}\label{rem:ZB}
As we pointed out in the introduction, the definition of the
  subspace $\calZ_{\beta}$  clearly depends on the coefficient $\kappa$,
  since $\beta$ depends on $\kappa$. Such dependence is often also
  seen in algebraic multigrid analysis, where the coarse spaces
  depend on the operator at hand.  They are in fact explicitly constructed 
  in this way, the aim being to increase robustness of the methods.
\end{remark}

In the proof of Proposition~\ref{teo0} above, we have introduced the basis
in both $V_{h}^{CR}$ and $\calZ_{\beta}$. The canonical Crouzeix-Raviart
basis functions  $\{\varphi_{e}^{CR}\}_{e\in \Eho}$ are continuous at the 
mass centers $m_e$ of the faces $e\in \Eho$.
The basis $\{\psi^{z}_{e}\}_{e\in \Eh}$ in $\calZ_{\beta}$
consists of piecewise $\mathbb{P}^{1}$ functions, which are discontinuous 
across the faces in $\Eh.$ 
In fact, for any $z\in \calZ_{\beta}$ such that
$z=\sum_{e\in \Eh} z_e\psi^{z}_{e}$ with $z_e\in \Reals{}$,  
we have
$$(\jump{z}\n^+)(m_{e^{\prime}}) =z_{e^{\prime}},\qquad\forall e^{\prime}\in \Eh.$$
To see this,  evaluating the jump of $z$ at  $m_{e^{\prime}}$ gives
\begin{eqnarray*}
(\jump{z}\n^+)(m_{e^{\prime}}) 
&=&\sum_{e\in \Eh} z_e
(\jump{\psi^{z}_{e}}\n^+)(m_{e^{\prime}})
=z_{e^{\prime}} (\jump{\psi^{z}_{e^{\prime}}}\n^+)(m_{e^{\prime}})\nonumber\\
&=&
\left\{
\begin{array}{ll}
z_{e^{\prime}} (\beta_{e^{\prime}}- (\beta_{e^{\prime}}-1))=
z_{e^{\prime}}, &\quad e^{\prime}\in
\Eho, \\
z_{e^{\prime}}, &\quad e^{\prime}\in
\Ehb. 
\end{array}
\right. 
\end{eqnarray*}
This relation will also be used later to obtain uniform diagonal preconditioners for the restrictions of
$\calA(\cdot,\cdot)$ and $\calA_0(\cdot,\cdot)$ on $\calZ_{\beta}$. 
\begin{remark}
  For mixed boundary value problems, that is, $\partial\Omega$ contains both Neumann boundary $\Gamma_{N}\neq \emptyset$ and Dirichlet boundary $\Gamma_{D}$ with $\partial\Omega=\Gamma_D\cup\Gamma_N$,  the definition of the basis functions on the boundary faces [see \eqref{basis-zb}] needs to be changed as:
\begin{equation}\label{basis-zbcr}
\begin{array}{lclll}
\phi_{e}^{CR}(x)&=&\varphi_{e,\K}(x),&\quad
e=\partial\K\cap\Gamma_N,&\quad \mbox{for all}\quad 
x\in T,\\
\psi_{e}^{z}(x)&=&\varphi_{e,\K}(x), &\quad e = \partial\K\cap \Gamma_D,&
\quad \mbox{for all}\quad  x\in T.
\end{array}
\end{equation}
Thus, in case $\Gamma_N\neq \emptyset$ 
the dimension of $V_{h}^{CR}$ is
increased (by adding to it functions that correspond to degrees of
freedom on $\Gamma_N$) and the dimension of $\calZ_\beta$ is decreased
accordingly. Clearly things balance out correctly: the identity
$V_{h}^{DG}=V_{h}^{CR}\oplus \calZ_{\beta}$ holds, and also the analysis
carries over with very little modification. 
\end{remark}
Next lemma is a simple but key observation used in the design of
efficient solvers. 
\begin{lemma}\label{le:orto}
  Let $\calA_0(\cdot,\cdot )$ be the bilinear form
  defined in~\eqref{ipA0}. Then,
\begin{equation}\label{orto:1}
\calA_0(v,z)=0 \qquad \forall\, v \in V_{h}^{CR},\quad \forall\, z \in \calZ_{\beta} .
\end{equation}
Furthermore if $\calA_0(\cdot,\cdot )$ is symmetric (and positive
definite) then
 the decomposition \eqref{splitting0} is
$\calA_{0}$-orthogonal, namely, $V_{h}^{CR}\,\, \perp_{\calA_0}\,\, \calZ_{\beta}$.
\end{lemma}
\begin{proof}
From the weighted-residual form of $\calA_0(\cdot,\cdot)$ given
  in~\eqref{ip:res0}, for all $v\in V_{h}^{CR}$, and all $z\in
  \calZ_{\beta}$ we easily obtain
\begin{equation*}
 \begin{aligned}
 \calA_0(v,z)&=(-\nabla\cdot (\kappa\nabla v), z)_{\Th} +\langle \jump{\kappa\nabla  v},
 \av{z}_{1-\beta_{e}}\rangle_{\Eho} +\langle \jump{v},
 \calP^{0}_{e}(\mathcalB_{1}(z))\rangle_{\Eh}=0. 
\end{aligned}
 \end{equation*}
 In the equation above, the first term is zero due to the fact that $v\in V_h^{CR}$,
 so $v$ is linear in each $\K$, and the coefficient $\kappa \in \mathbb{P}^{0}(\K)$. Last term vanishes (independently
 of the choice of $\theta$, or equivalently the choice of
 $\mathcalB_1(v)$), because $\langle \jump{v},
 \calP^{0}_{e}(\mathcalB_{1}(z))\rangle_{\Eh}=0$, thanks to the definition \eqref{defCR} of $V_{h}^{CR}$.  The second term vanishes from the definition of
 $\calZ_{\beta}$ (since $\jump{\kappa\nabla v}$ is constant on each
 $e\in \Eho$). Moreover, in the case when $\calA_0(\cdot ,\cdot)$ is
 symmetric and positive definite we have that
 $\calA_0(v,z)=\calA_0(z,v)$, for all $v\in V_{h}^{CR}$ and for all
 $z\in \calZ_{\beta}$. Thus, for the symmetric method
 $\calA_0(\cdot,\cdot)$, the spaces $V_{h}^{CR}$ and $\calZ_{\beta}$
 are indeed $\calA_0$-orthogonal. The proof is complete.
\end{proof}

\section{Solvers for IP($\beta$)-0 methods}
\label{sec:solvers}

In this section we show how Proposition \ref{teo0} and Lemma
\ref{le:orto} can be used in the design and analysis of uniformly
convergent iterative methods for the IP($\beta$)-0 methods.
We follow the ideas and analysis introduced in
\cite{Ayuso-de-DiosBZikatanovL-2009aa} 
and point out the differences.
We first consider the approximation to problem \eqref{eqn:model} with
$\calA^{DG}(\cdot,\cdot )=\calA_0(\cdot,\cdot)$.
To begin, 
let $A_0$ be the discrete operator defined by $(A_0 u,w) =
\calA_0(u,w)$ and let $\mathbb{A}_0$ be its matrix representation in
the new basis \eqref{basis-CR} and \eqref{basis-z}. We denote by $\mathbf{u}=[\mathbf{z},\mathbf{v}]^{T}$, $\mathbf{f}=[\mathbf{f_z},\mathbf{f_v}]^{T}$ the vector representation of the unknown function $u$ and of the right
hand side $f$, respectively, in this new basis. A simple consequence
of Lemma \ref{le:orto} is that the matrix $\mathbb{A}_0$ (in this
basis) has block lower triangular structure:
 \begin{equation}\label{Acal0}
\mathbb{A}_0= \left[ \begin{array}{cccc}
&\mathbb{A}^{zz}_0 & \mathbf{0} &\\
& \mathbb{A}^{vz}_0 & \mathbb{A}^{vv}_0 &
\end{array}\right], 
\end{equation}
where $\mathbb{A}^{zz}_0, \mathbb{A}^{vv}_0$ are the matrix
representation of $A_0$ restricted to the subspaces $\calZ_{\beta}$
and $V_{h}^{CR}$, respectively, and $\mathbb{A}^{vz}_0$ is the matrix
representation of the term that accounts for the coupling (or
non-symmetry) $\calA_0(\psi^z,\varphi^{CR})$.  
As remarked earlier, for SIPG$(\beta)$-0, the
stiffness matrix $\mathbb{A}_0$ is block-diagonal.

Figure~\ref{fig:domain2d} gives a 2D example, with two squares $\Omega_{1}=[-0.5,0]^{2}$
and $\Omega_{2}=[0,0.5]^{2}$ inside the domain $\Omega= [-1,1]^{2}.$
We set the coefficients $\kappa(x) =1$ for all $x\in
\Omega_{1}\cup\Omega_{2}$ and $\kappa(x) =10^{-3}$ for $x\in
\Omega\setminus(\Omega_{1}\cup \Omega_{2})$.
\begin{figure}[htbp]
        \includegraphics[width=0.4\textwidth]{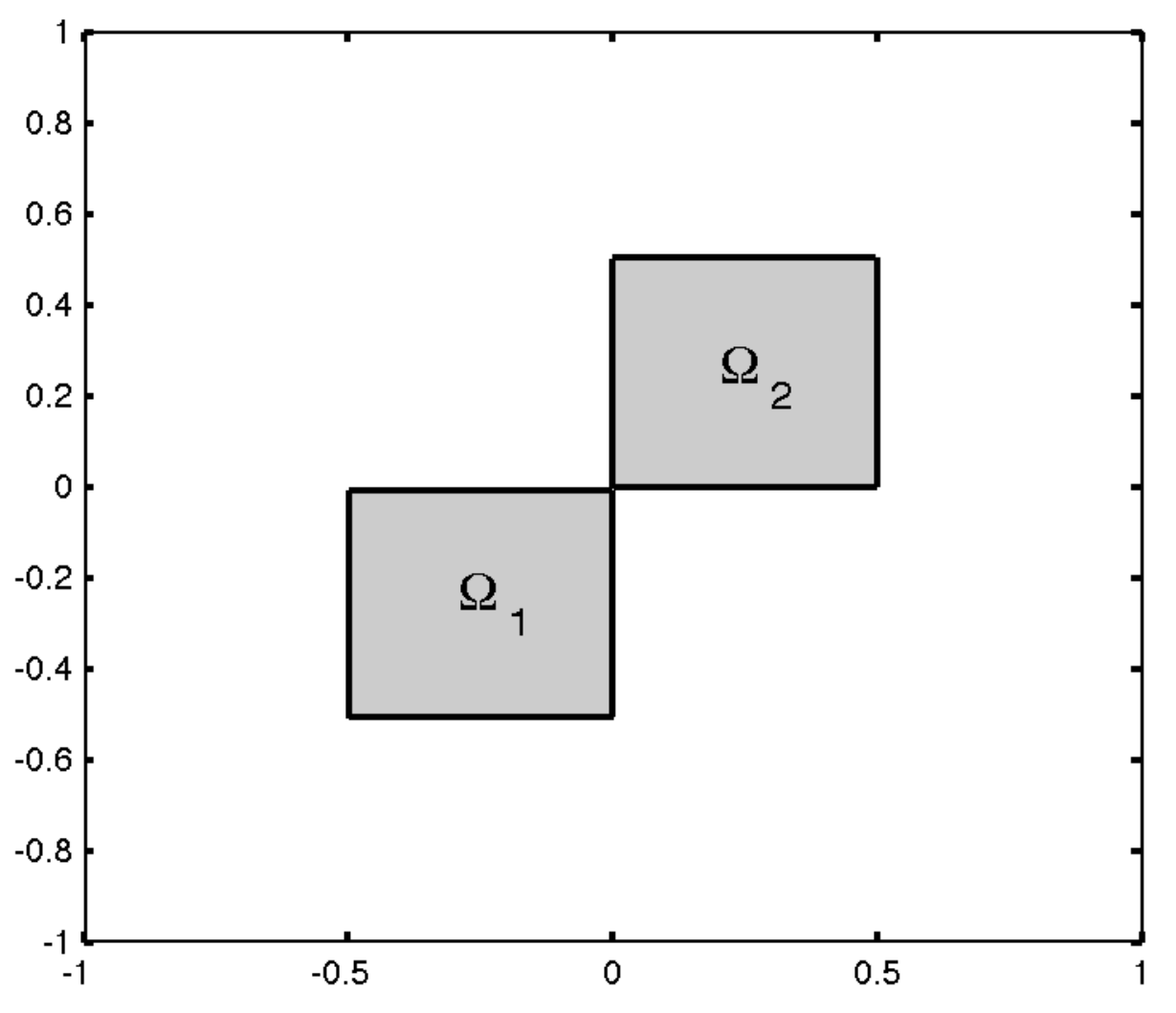}\hskip 0.8cm
        \includegraphics[width=0.47\textwidth]{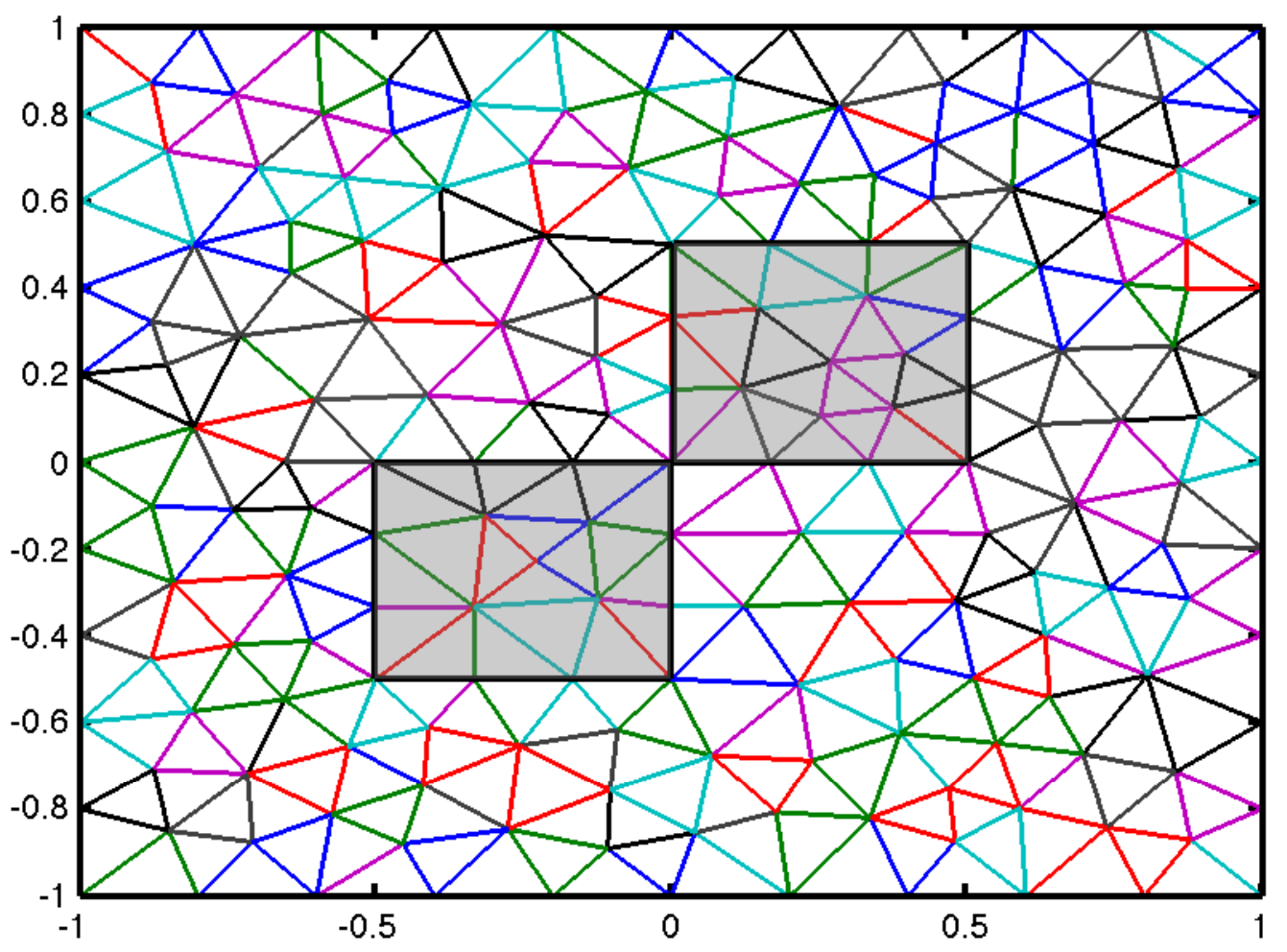}
\caption{\it Computational domain and unstructured mesh.\label{fig:domain2d}}
\end{figure}
Figures~\ref{fig0b}  and ~\ref{fig0a} show  the sparsity patterns of the 
IP$(\beta)$-0 methods with standard nodal basis and the basis~\eqref{basis-CR}-\eqref{basis-z}, respectively.

\begin{figure}[!htp]
        \begin{center}
                \includegraphics[width=4.2cm]{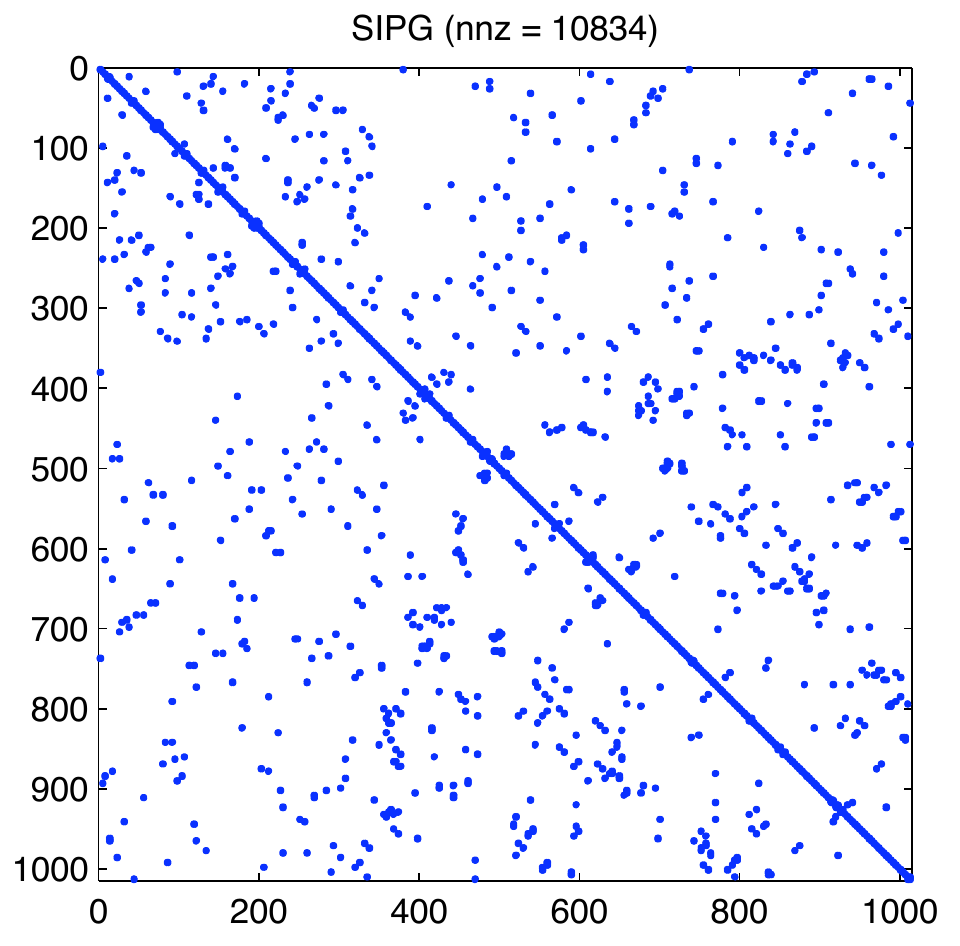}
        \hskip 0.06cm
                \includegraphics[width=4.2cm]{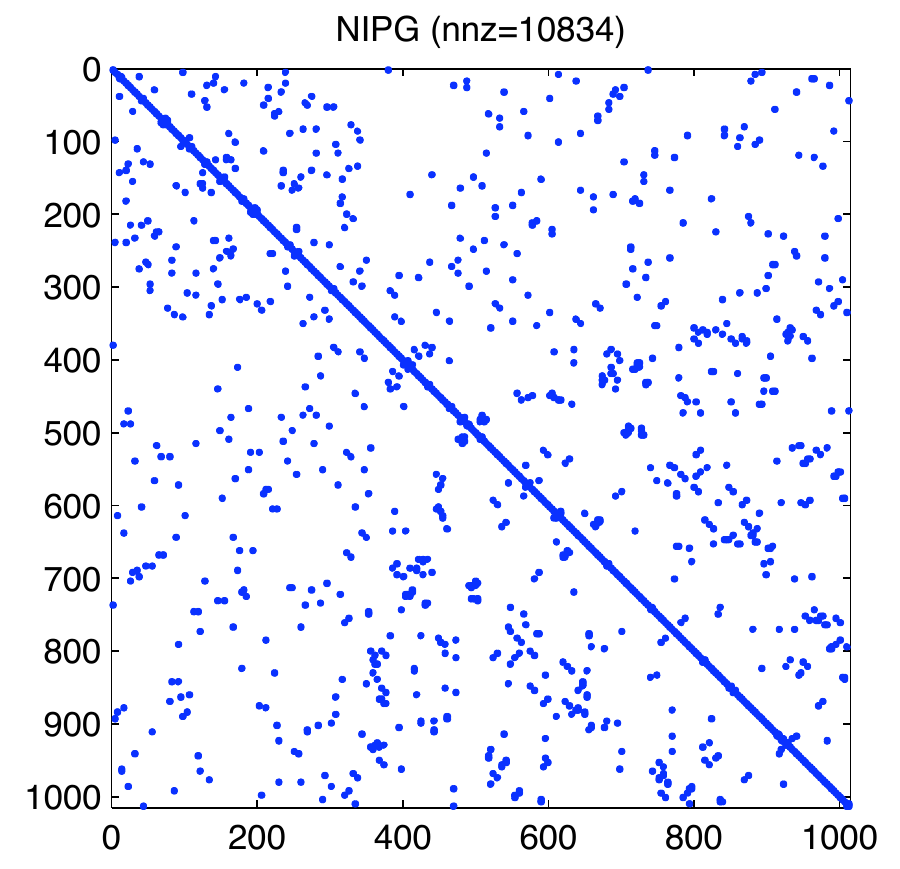}
            \hskip 0.07cm
             \includegraphics[width=4.2cm]{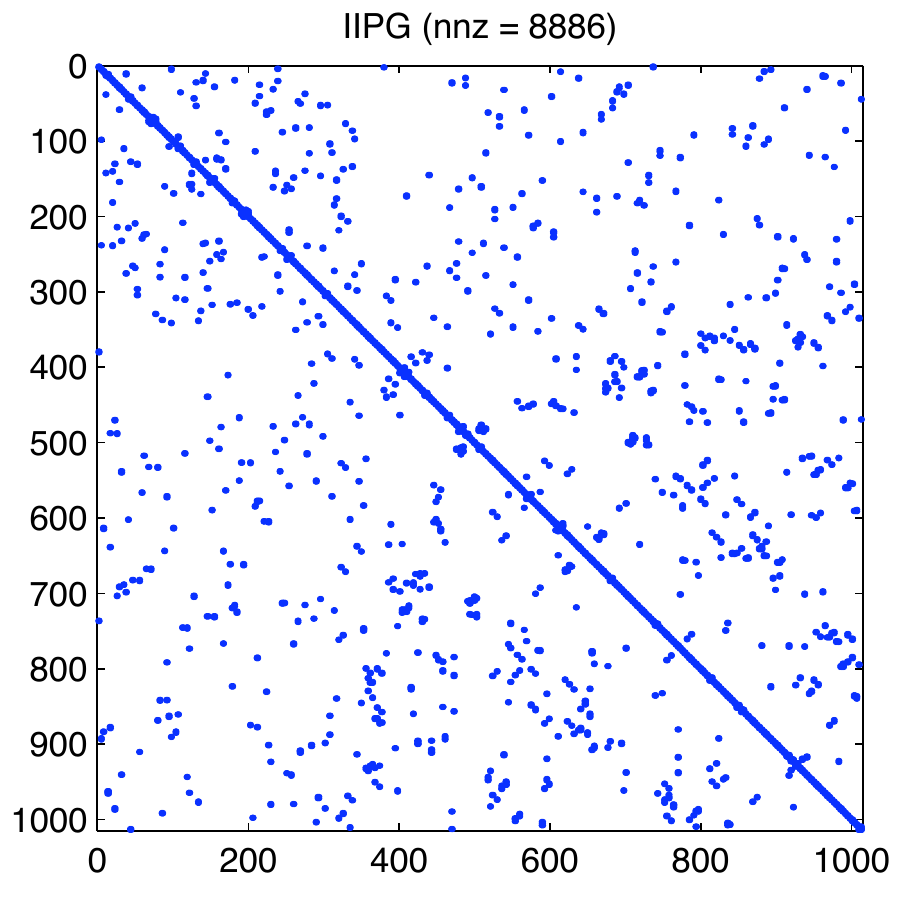}
        \end{center}
             \caption{Non-zero pattern of the matrix representation in the standard nodal basis of the operators associated with IP($\beta$)-0 methods. From left to right: SIPG, NIPG and IIPG methods.}
      \label{fig0b}
          \end{figure}
\begin{figure}[!htp]
        \begin{center}
                \includegraphics[width=4.2cm]{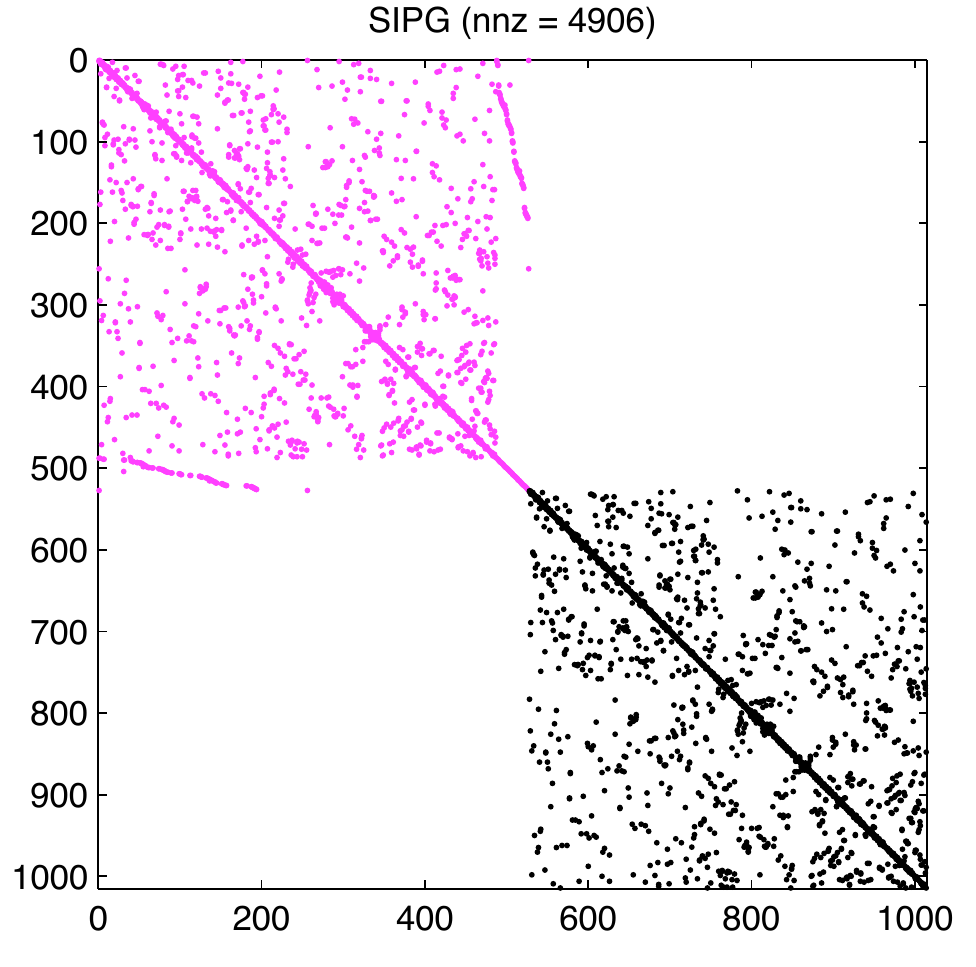}
        \hskip 0.06cm
                \includegraphics[width=4.2cm]{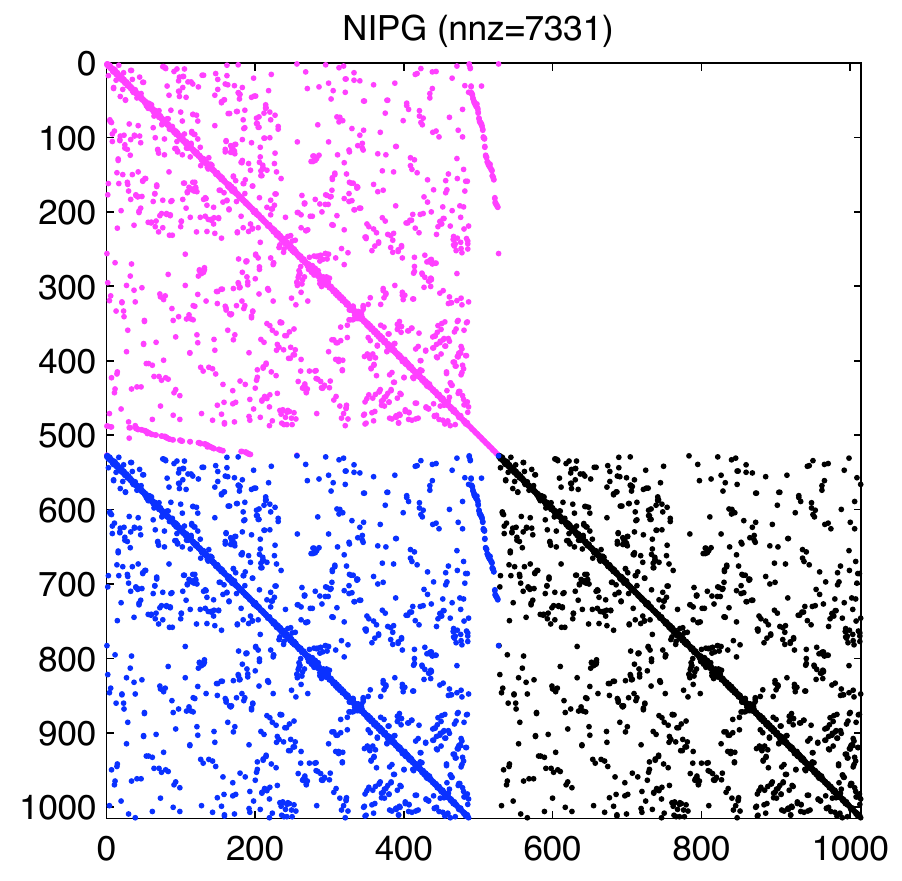}
            \hskip 0.07cm
             \includegraphics[width=4.2cm]{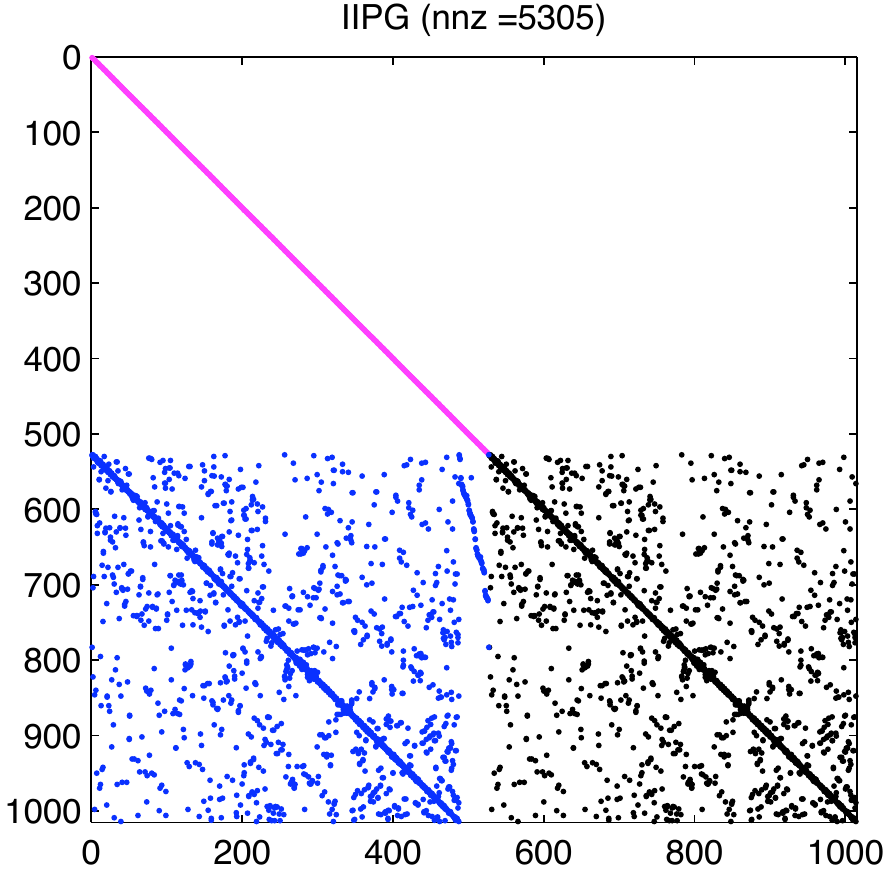}
        \end{center}
             \caption{Non-zero pattern of the matrix representations according to the basis splitting~\eqref{splitting0} of the operator associated with IP($\beta$)-0 methods, i.e., $\mathbb{A}_0$. From left to right: SIPG, NIPG and IIPG methods.}
      \label{fig0a}
\end{figure}

Clearly, as in the constant coefficient case, a simple algorithm based 
on a block version
of forward substitution provides an exact solver for the solution of
the linear systems with coefficient matrix $\mathbb{A}_0$. A formal
description of this block forward substitution is given as the next Algorithm.
\begin{algorithm}[Block Forward Substitution]\label{algo0}
~
\begin{itemize} 
\item[1.] Find $z\in \calZ_{\beta}$ such that $\calA_0(z,\psi) = (f,\psi)_{\Th}$  for
  all $\psi\in \calZ_{\beta}$\;

\item[2.] Find $v\in V_{h}^{CR}$ such that $\calA_0(v,\varphi) =
  (f,\varphi)_{\Th}-\calA_0(z,\varphi)$  for all $\varphi\in V_{h}^{CR}$\;

\item[3.] Set $u=z+v$\;
\end{itemize}
\end{algorithm}

The above algorithm requires the solution of
$\calA_0(\cdot,\cdot)$ on $\calZ_{\beta}$ (Step 1. of the algorithm)
and the solution of $\calA_0(\cdot,\cdot)$ on $V_h^{CR}$ (Step~2.\ of
the algorithm). Unlike the situation in
\cite{Ayuso-de-DiosBZikatanovL-2009aa}, due to the jump coefficient
in \eqref{eqn:model}, the solution on $V_{h}^{CR}$ is more involved, and
therefore we postpone its discussion and analysis until Section
\ref{sec:crsolver}.  We next discuss the solution on $\calZ_{\beta}$.


\subsection{Solution on $\calZ_{\beta}$} 
In this section we describe the main properties of the IP($\beta$)-0 methods when restricted to the $\calZ_{\beta}$, which will in turn indicate how the solution of Step~1.\ of Algorithm \ref{algo0} can be efficiently done. 

The first result in this subsection establishes the symmetry of the restrictions of
the bilinear forms of the IP($\beta$)-0 methods to $\calZ_{\beta}$.
\begin{lemma}\label{symZZ}
  Let $\calA_0(\cdot,\cdot)$ be the bilinear form of a 
   IP($\beta$)-0 method as defined in \eqref{ipA0}. Then, the restriction to $\calZ_{\beta}$ of 
  $\mathcal{A}_{0}(\cdot,\cdot)$ is
  symmetric. Namely, for $\theta=-1,\, 0,\, 1$, we have
\begin{equation*}
\calA_0(z,\psi)=\calA_{0}(\psi,z) \qquad \forall\,  z ,\psi \in
  \calZ_{\beta}\;.
\end{equation*}
\end{lemma} 
\begin{proof}
  If $\theta=-1$ there is nothing to prove, since in this case both
  bilinear forms are symmetric. Hence we only consider the cases
  $\theta=0$ or $\theta=-1$. Integrating by parts and using the fact that
  $z\in\calZ_{\beta}$ and $\psi \in \calZ_{\beta}$ are linear on each
  element $\K$ shows that
\begin{equation*}
0=(-\nabla \cdot(\kappa\nabla \psi),\nabla z)_{\Th}=(\kappa\nabla \psi,\nabla z)_{\Th}-\langle \av{\kappa \nabla \psi}_{\beta_e}, \jump{z}\rangle_{\Eh}-\langle \jump{\kappa\nabla  \psi}, \av{z}_{1-\beta_{e}}\rangle_{\Eho}\;.
\end{equation*}
Hence, from the definition \eqref{defZ} of the $\calZ_{\beta}$ space, it follows that
\begin{equation}\label{perlinear}
(\kappa\nabla \psi,\nabla z)_{\Th}=
\langle \av{\kappa \nabla \psi}_{\beta_e}, \jump{z}\rangle_{\Eh}
=\langle \av{\kappa \nabla z}_{\beta_e}, \jump{\psi}\rangle_{\Eh}, \quad \forall\, z, \psi \in \calZ_{\beta} .
\end{equation}
Substituting the above identity in the definition of the bilinear form
\eqref{ipA0} then leads to
\begin{align*}
\calA_0(z,\psi)&= \theta\langle \jump{z},\av{\kappa\nabla \psi}_{\beta_e}\rangle_{\Eh}  +\langle \calP^{0}_{e}(\jump{z}),\kappa_{e} \jump{\psi} \rangle_{\Eh} \qquad \qquad&&\\
&=\theta (\kappa\nabla \psi,\nabla z)_{\Th} +\langle \calP^{0}_{e}(\jump{\psi}), \kappa_{e}\jump{z}\rangle_{\Eh}=\calA_0(\psi,z). &&
\end{align*}
This shows the symmetry of $\calA_0(\cdot,\cdot)$ on
$\calZ_{\beta}$, regardless the value of $\theta$.
\end{proof}

We now study the conditioning of the bilinear form
 $\calA_0(\cdot,\cdot)$ on $\calZ_{\beta}$.
For all $z\in \calZ_{\beta}$, and for all
  $\phi\in \calZ_{\beta}$ with
\begin{equation*}
z=\sum_{e\in \Eh}z_e\psi_{e}^{z} \,\in\,
\calZ_{\beta},\quad\mbox{and}\quad
\phi=\sum_{e\in \Eh}\phi_e\psi_{e}^{z} \,\in\, \calZ_{\beta}.
\end{equation*}
we introduce a weighted scalar product
$(\cdot,\cdot)_{\ast}:\calZ_{\beta}\times
\calZ_{\beta}\mapsto\Reals{}$ and the corresponding norm $\|\cdot\|_{
  \ast}$, defined as follows
\begin{equation}\label{weighed-l2}
(z,\phi)_{ \ast} := \sum_{e\in \Eh}
\frac{|e|}{h_e}\kappa_{e} 
z_e\phi_e\;, \quad \|z\|^2_{ \ast} := (z,z)_{\ast}.
\end{equation}
Observe that the matrix representation  (in the basis given in~\eqref{basis-z}) of the above weighted scalar product is in fact a diagonal
matrix.
The next result shows that the restriction of
$\calA_0(\cdot,\cdot)$ to $\calZ_{\beta}$ is spectrally
equivalent to the weighted scalar product $(\cdot,\cdot)_{\ast}$ defined in \eqref{weighed-l2} and
therefore its matrix representation $\mathbb{A}_0^{zz}$ is spectrally equivalent to a
diagonal matrix.

\begin{lemma}\label{le:diag0}
Let $\calZ_{\beta}$ be the space defined in \eqref{defZ}. Then, the following estimates hold
\begin{equation}
\label{diag0:0}
\|z\|_{ \ast}^{2}\lesssim \calA_0(z,z) \lesssim  \|z\|_{ \ast}^{2} \qquad \forall\, z\in
\calZ_{\beta}\;.
\end{equation}

\end{lemma}
\begin{proof}
  Let us fix $z\in \calZ_{\beta}$, $z=\sum_{e\in \Eh}z_e\psi_{e}^{z}$.
  From~the definition of $\calP^{0}_e(\jump{z})$, it is
  immediate to see that
\[
\|\calP^{0}_e(\jump{z})\|_{0,e}^{2}=|e|z_e^{2}.
\]
Thus, we have that
\begin{equation}\label{a-provar}
\sum_{e\in\Eh} \kappa_{e}
h_e^{-1}\|\calP^{0}_e(\jump{z})\|_{0,e}^{2} = \sum_{e\in\Eh}
\kappa_{e}\frac{|e|}{h_e}z^2_e = \| z\|_{*}. 
\end{equation} 
To show the upper bound in \eqref{diag0:0}, we notice that
 \eqref{perlinear} together with \eqref{cota-a} and the standard trace and inverse  inequalities gives
 \begin{align*}
\sum_{\K\in \Th} \kappa_{\K}\|\nabla z\|^{2}_{0,\K} &=(\kappa\nabla z,\nabla z)_{\Th}= \langle \av{\kappa \nabla z}_{\beta_e}, \jump{z}\rangle_{\Eh}=\langle \kappa_{e}\av{ \nabla z}, \calP^{0}_e(\jump{z})\rangle_{\Eh} &&\\
&\lesssim \left(\sum_{\K\in\Th} \kappa_{\K}\|\nabla z\|^{2}_{0,\K}\right)^{1/2}\left(\sum_{e\in\Eh} \kappa_{e}\|h_e^{-1/2}\calP^{0}_e(\jump{z})\|_{0,e}^{2}\right)^{1/2}\;, &&
\end{align*} 
and therefore by \eqref{a-provar}, 
\begin{equation}\label{cota0}
\sum_{\K\in \Th} \kappa_{\K}\|\nabla z\|^{2}_{0,\K} \lesssim \sum_{e\in\Eh}
\kappa_{e}\|h_e^{-1/2}\calP^{0}_e(\jump{z})\|_{0,e}^{2}=\|z\|_{ \ast}^2.
\end{equation}
Since $z\in \calZ_{\beta}$ was arbitrary, 
we have that $\calA_0(z,z)\lesssim \|z\|_{ \ast}^2$  for all $z\in
\calZ_{\beta}$.
This proves the upper bound in~\eqref{diag0:0}.

To prove the lower bound, we use the coercivity estimate  \eqref{eq:coer} for the
bilinear form $\calA_0(\cdot,\cdot)$ in the energy norm
$\triplenorm{\cdot}_{DG0}^{2}$ [see~\eqref{def-normN}]. For all $z\in \calZ_{\beta}$ we have
\begin{eqnarray*}
\calA_0(z,z)& \gtrsim& \triplenorm{z}_{DG0}^{2}=
\sum_{\K\in \Th} \kappa_{\K}\|\nabla z\|^{2}_{0,\K}
+\sum_{e\in\Eh}
\kappa_{e}\|h_e^{-1/2}\calP^{0}_e(\jump{z})\|_{0,e}^{2}\\
& \gtrsim&  \sum_{e\in\Eh}
\kappa_{e}\|h_e^{-1/2}\calP^{0}_e(\jump{z})\|_{0,e}^{2}=\|z\|^2_{ \ast}\;,
\end{eqnarray*}
which is the desired bound and gives~\eqref{diag0:0}.  \end{proof}

Last result guarantees that the linear systems on
$\mathcal{Z}_{\beta}$ can be efficiently solved by preconditioned CG
(PCG) with a diagonal preconditioner. 
As a corollary of the result in Lemma~\ref{le:diag0},
the number of PCG iterations will be independent of both the 
mesh size and the variations in the PDE coefficient.

We end this section by showing that in the particular case of the IIPG($\beta$)-0 
method, the matrix representation of $\calA_0(\cdot,\cdot)$ restricted to $\calZ_{\beta}$ is
in fact a diagonal matrix. (See the rightmost figure in Fig. \ref{fig0a}).
\begin{lemma}\label{lem:ZZ}
  Let $\calA_0(\cdot,\cdot)$ be the bilinear form of the non-symmetric
  IIPG$(\beta)$-0 method \eqref{ipA0} with $\theta=0$. Let
  $\{\psi^{z}_{e}\}_{e\in \Eh}$ be the basis for the space
  $\calZ_{\beta}$ as defined in \eqref{basis-z}. Let
  $\mathbb{A}_0^{zz}$ be the matrix representation in this basis of
  the restriction to the subspace $\calZ_{\beta}$ of the operator
  associated to $\calA_0(\cdot,\cdot)$. Then, $\mathbb{A}_{0}^{zz}$ is
  diagonal.\end{lemma}
\begin{proof}
  Note that from the definition \eqref{ipA0} of the method
  ($\theta=0$) together with \eqref{perlinear} we have
\begin{align}\label{nume00}
\calA_0(z,\psi)&=(\kappa\nabla z,\nabla \psi)_{\Th} -\langle \av{ \nabla z}_{\beta_{e}}, \jump{\psi}\rangle_{\Eh}+\langle \alpha h_e^{-1}  \kappa_{e}\calP_e^{0}(\jump{ z}), \calP^{0}_{e}(\jump{\psi})\rangle_{\Eh}  &&\nonumber \\
&=\langle  \alpha h_e^{-1}  \kappa_{e}\calP_e^{0}(\jump{ z}), \calP^{0}_{e}(\jump{\psi})\rangle_{\Eh}, \quad \forall\, z, \psi \in \calZ_{\beta}\; . &&
\end{align}
Let $\{\psi^{z}_{e}\}_{e\in \calE_h}$ be the basis functions \eqref{basis-z}. To prove that $\mathbb{A}_0^{zz}$ is diagonal it is enough to show that for the basis functions  \eqref{basis-z}, the following relation holds:
\begin{equation}\label{numerito3}
\mathcal{A}_{0}(\psi^{z}_{e},\psi^{z}_{e'})=c_{e}\delta_{e,e'}, \qquad c_{e}\ne 0, \qquad \forall\, e\in \Eh\;,
\end{equation}
where $\delta_{e,e'}$ is the delta function associated with the edge/face $e$. We now show \eqref{numerito3}.  Observe that the supports of $\psi^{z}_{e}$ and $\psi^{z}_{e'}$  have empty intersection unless $e,e' \subset
\K$ for some $\K\in \Th$.    Let $\K \cap
\partial\Omega=\emptyset$ be an interior element, then from \eqref{nume00} and the mid-point integration rule, we have
\begin{align*}
  \mathcal{A}_{0}(\psi^{z}_{e},\psi^{z}_{e'})&= \alpha h_e^{-1} \int_e \kappa_{e}\calP^{0}_{e}(\jump{\psi^{z}_{e}})\calP^{0}_{e}(\jump{\psi^{z}_{e'}}) ds=  \alpha h_e^{-1} \kappa_{e} [2\psi^{z}_{e}(m_{e})] [2\psi^{z}_{e'}(m_{e})] &&\\
  &= 4\alpha h_e^{-1} \kappa_{e}  \delta_{e,e'}, 
\qquad e,e'\subset \partial\K, \quad e,e' \in \Eho\;,
\end{align*}
which shows \eqref{numerito3} for interior edges with $c_e=4\alpha
h_e^{-1} \kappa_{e}$. For boundary edges/faces the considerations are 
essentially the same and therefore omitted.  The proof is complete
since the relation~\eqref{numerito3} readily implies that the off-diagonal
terms of $\mathbb{A}_0^{zz}$ are zero.
\end{proof}

\section{Robust Preconditioner on $V_{h}^{CR}$}
\label{sec:crsolver}
In this section, we develop efficient and robust (additive) two-level and multilevel preconditioners for the solution  of the IP($\beta$)-0 methods in the CR space (cf. Step~2 of algorithm \ref{algo0}).  We first review a few preliminaries and tools that will be needed for the convergence analysis. We then define the two-level preconditioner and provide the convergence analysis. The last part of the section contains the construction and convergence analysis of the multilevel preconditioner.

From the definition \eqref{defCR} of the $V_h^{CR}$ space, it follows that the restriction of $\calA_0(\cdot,\cdot)$ to $V_{h}^{CR}$ reduces to the classical $\mathbb{P}^{1}$-nonconforming finite
element discretization of \eqref{eqn:model}:
\begin{eqnarray}\label{prob:cr}
  \mbox{Find } u\in V^{CR}_{h} :& \calA_0(u, w) =(\kappa\nabla u, \nabla w)_{\Th}=(f, w), & \forall\, w\in V^{CR}_{h}. 
 \end{eqnarray}
We denote $A_{0}^{CR}$ as the operator induced by \eqref{prob:cr}.
For the analysis in this section, we will need the following semi-norms
and norms for any $v\in V_h^{CR}$:
\begin{align}
|v|^{2}_{1,h,\kappa}&:=\sum_{\K\in \Th} \kappa_{\K} \|\nabla v\|_{0,\K}^{2}\;,  \qquad |v|^{2}_{1,h,\Omega_i}:= \sum_{\K\in \Th\;,\, \K\subseteq  \O_{i} } \|\nabla v\|_{0,\K}^{2}\;, \label{def-norms10}&&\\
\|v\|^{2}_{0,\kappa}:&=\sum_{i=1}^{M} \kappa\big|_{\O_i} \|v\|_{0,\O_i}^{2}\;, 
\qquad\|v\|^{2}_{1,h,\kappa}:=\|v\|^{2}_{0,\kappa}+ |v|^{2}_{1,h,\kappa}\;.\label{def-norms00}&&
\end{align}
Since \eqref{prob:cr} is a symmetric and coercive problem, from the classical
theory of PCG we know that the convergence rates of the iterative
method for $A_0^{CR}$ with preconditioner, say $B$,  are fully
determined, in the worst case scenario, by the condition number of the
preconditioned system: $\mathcal{K}(BA_0^{CR})$.  However, if the spectrum of $BA_0^{CR}$,
$\sigma(BA_0^{CR})$ happens to be divided in two sets:
$\sigma(BA_0^{CR})=\sigma_0(BA_0^{CR})\cup \sigma_1(BA_0^{CR})$, where
$\sigma_0(BA_0^{CR})=\{\lambda_1,\ldots ,\lambda_{m}\}$ contains
all of the very small (often referred to as ``bad'') eigenvalues,
and the remaining eigenvalues (bounded above and below) are contained in
$\sigma_1(BA_0^{CR})=\{\lambda_{m+1}, \ldots , \lambda_{n_{CR}}\}$,
that is, $\lambda_j \in [a,b]$ for $j=m+1,\ldots, n_{CR}$, with
$n_{CR}=\mbox{dim}(V_h^{CR})=n_{E}-n_{BE}$, i.e. the number of interior edges, then the error
at the $k$-th iteration of the PCG algorithm is bounded by 
(see e.g.~\cite{Axelsson.O1994,Hackbusch.W1994,Axelsson.O2003}):
\begin{equation}
\label{eqn:CG}
\frac{\|u-u_k\|_{1,h,\kappa}} {\|u-u_0\|_{1,h,\kappa}}\le 2(\mathcal{K}(BA_0^{CR})-1)^{m}
\left(\frac{\sqrt{b/a}-1}{\sqrt{b/a}+1}\right)^{k-{m}}\;.
\end{equation}
The above estimate indicates that if $m$ is not large (there are 
only a few very small eigenvalues) then the \emph{asymptotic} convergence rate of the resulting PCG
method will be dominated by the factor
$\frac{\sqrt{b/a}-1}{\sqrt{b/a}+1}$, i.e. by $\sqrt{b/a}$ where
$b=\lambda_{N}(BA_0^{CR})$ and $a=\lambda_{m+1}(BA_0^{CR})$.  The
quantity $(b/a)$ which determines the asymptotic convergence rate is
often called \emph{effective condition number}.
This is precisely the situation in the
case of problems with large jumps in the coefficient $\kappa$. In fact,  for a conforming FE approximation to \eqref{eqn:model}  it has 
been observed in \cite{GrahamI_HaggerM-1999aa, XuJZhuY-2008aa} that
the spectrum $\sigma(BA_0^{CR})$ might contain a few very small
eigenvalues, which result in an extremely large value of
$\mathcal{K}(BA_0^{CR})$. Nevertheless,  they seem to have very little influence on the efficiency and overall performance of the PCG method. Therefore, it is natural to study the {\it asymptotic} convergence in this case, which as mentioned above is determined by the \emph{effective condition
  number}:
 \begin{definition}
\label{def:effcond}
Let $V$ be a real $N$-dimensional Hilbert space, and
$\mathfrak{A}:V\to V$ be a symmetric positive definite linear
operator, with eigenvalues $0 < \lambda_{1} \le \cdots \le
\lambda_{N}$. The \emph{$m$-th effective condition number} of
$\mathfrak{A}$ is defined by
    \begin{equation*}
	 \mathcal{K}_{m}(\mathfrak{A}) := \frac{\lambda_{N}(\mathfrak{A})}{\lambda_{m+1}(\mathfrak{A})}.
	 \end{equation*}
\end{definition}
Below, we will introduce the two-level and multilevel preconditioners, and study in detail the spectrum of the preconditioned systems. In particular, we gave estimates on both condition numbers and the effective condition numbers, which indicates the \emph{pre-asymptotic} and {\it asymptotic} convergence rates in \eqref{eqn:CG} of the PCG algorithms. 
\subsection{Two-level preconditioner for $\calA_0(\cdot,\cdot)$ on
 $V_{h}^{CR}$}\label{sec:2level}
In this subsection, we construct a two-level additive preconditioner,
which consists of a standard pointwise smoother (Jacobi, or
Gauss-Seidel) on the nonconforming space $V_{h}^{CR}$ plus a coarse solver on a (possibly coarser) conforming space
$\Vc:=\{v\in H_{0}^{1}(\Omega): v|_{T} \in \mathbb{P}_{1}(T),\;\;
\forall T\in \mathcal{T}_{\h}\}$. Here, $\mathcal{T}_{\h}$ refers to a possibly coarser partition such that  $\mathcal{T}_{\h} \subseteq \mathcal{T}_{h}$; that is for $\h = h$, $\mathcal{T}_{\h}$ is the same as $\mathcal{T}_{h}$, while for $\h =H>h$ the partitions are nested and $\mathcal{T}_{h}$ could be regarded as a refinement of $\mathcal{T}_{\h}$.  Observe that $\Vc$ is a proper
subspace of $V_{h}^{CR}$. To define the two-level preconditioner, we
consider the following (overlapping) space decomposition of $\Vcr$:
  \begin{equation}
	\label{eqn:space-decomp}
	\Vcr=\Vcr + \Vc.
\end{equation}

On $\Vc$ we consider the standard conforming
$\mathbb{P}^{1}$-approximation to \eqref{eqn:model}: Find $\chi \in
\Vc$ such that
\begin{equation}\label{a:conf}
  \calA_0(\chi,\eta)
=a(\chi,\eta)=\int_{\O} \kappa \nabla \chi \cdot \nabla \eta dx = (f,\eta), \quad \forall\, \eta \in \Vc\;.
\end{equation}
The bilinear form in \eqref{a:conf} defines a natural ``energy'' inner product, and induces the following weighted energy norm:
\begin{equation}\label{def-norms0}
  |\chi|^{2}_{1,\kappa, D} :=  \int_{D} \kappa|\nabla \chi |^{2} dx \;,\quad \forall\, \chi \in H^{1}(D),\quad  D \subset \Omega.
\end{equation}
For simplicity, we write  $|\chi|_{1,\kappa} =  |\chi|_{1,\kappa, \O}$ and denote by $A^{C}$ the operator associated to \eqref{a:conf}. We define the  two level preconditioner as:
\begin{equation}\label{eq:2Loperator}
B: \Vcr \mapsto \Vcr, \quad \quad
B:=R^{-1} + (A^{C})^{-1} Q^C,
\end{equation}
where $R^{-1}$ is the operator corresponding to a Jacobi or symmetric Gauss-Seidel smoother on $\Vcr$, and $Q^{C} : \Vcr \mapsto \Vc$ is the standard $L^{2}$-projection.  We refer to \cite{ahzz:tech} for further details on the matrix representation of the above preconditioner.

Next Theorem is the main result of this section, which establishes the convergence for the two-level preconditioner \eqref{eq:2Loperator}.

\begin{theorem}\label{teo1}
  Let $B$ be the multilevel preconditioner defined in \eqref{eq:2Loperator}, and $\varpi=\h/h$ be the ratio of the mesh sizes of $\cT_{\h}$ and $\cT_{h}$.   Then, the condition number $\mathcal{K}(BA_{0}^{CR})$ satisfies:
  	\begin{equation}\label{bound:AAA}
		\mathcal{K}(BA_{0}^{CR}) \le C_{0} \mathcal{J}(\kappa) \varpi^{2}\log(2\varpi)\;,
	\end{equation}
		where $\mathcal{J}(\kappa):= \max_{T\in \Th} \kappa_{T} /\min_{T\in \Th} \kappa_{T}$ is what we refer as the jump of the coefficient and $C_{0}>0$ is a constant independent of the coefficient $\kappa$ and the mesh size. Moreover, there exists an integer $m_{0}$ depending only on the distribution of the coefficient $\kappa$ such that the $m_{0}$-th effective
  condition number $\mathcal{K}_{m_{0}}(B A_0^{CR})$ satisfies:
	\begin{equation*}
	 \mathcal{K}_{m_{0}}(B A_0^{CR}) \le C_1 \varpi^{2}\log(2\varpi )\;,
	\end{equation*}
	where $C_{1}>0$ is a constant independent of the coefficient and mesh size. Hence, the convergence rate of the PCG algorithm can be bounded as
	\begin{equation}
\label{eqn:2LCG-rate}
\frac{|u-u_k |_{1,h,\kappa}}{|u-u_0 |_{1,h,\kappa}}\le 2 \left(C_{0}\mathcal{J}(\kappa) \varpi^{2}\log(2\varpi ) -1 \right)^{m_0}
\left(\frac{\sqrt{C_{1}}\varpi \log^{1/2}(2\varpi )-1}{\sqrt{C_1}\varpi \log^{1/2}(2\varpi ) +1}\right)^{k-{m_0}} \;.
\end{equation}
\end{theorem}

\begin{remark}
\label{rk:nolog}
We emphasize that for two-level preconditioners, since the ratio
$\varpi=\h/h$ is a fixed constant, the effective condition number 
$\mathcal{K}_{m_{0}} (B A_{0}^{CR})$ is bounded uniformly with respect to the coefficient variation and mesh size. Clearly, according to estimate \eqref{eqn:2LCG-rate}, the number of (pre-asymptotic) PCG iterations will depend on the constant $m_0$ (the number of floating subdomains; see~\eqref{def:I}). 
While such a bound could be a large overestimate (depending on
the coefficient distribution), it is sufficient for our purposes.
Since $m_0$ is fixed, the
asymptotic convergence rate in \eqref{eqn:2LCG-rate} is bounded
uniformly with respect to coefficient variation and mesh size. In
short, while the estimates given here might not be sharp
with regard to the pre-asymptotic PCG convergence, they are
asymptotically uniform with respect to the parameters of interest.
\end{remark}

We recall the following well known identity~\cite[Lemma 2.4]{XuJZikatanovL-2002aa}:
\begin{equation}\label{eq:additive-inverse}
(B^{-1} v , v) = \inf_{\chi\in \Vc} [\calR(v-\chi,v-\chi) + a(\chi,\chi)] \quad \forall v\in V_{h}^{CR}, 
\end{equation}
where $\calR(\cdot,\cdot)$ is the bilinear form associated with the
smoother  defined by $\calR(v,w) := (Rv,w)$  for any $w, v\in V_{h}^{CR}$. 
The proof of
Theorem~\ref{teo1} amounts to showing a smoothing property for 
$\calR(\cdot, \cdot)$ and the stability of the decomposition given in \eqref{eqn:space-decomp}.
Next Lemma establishes the former; the latter is contained in next subsection.
\begin{lemma}\label{lm:smoothCR}
  Let $\mathcal{R}(\cdot,\cdot)$ be the bilinear form associated to Jacobi, or symmetric Gauss-Seidel smoother. Then we  have the following estimates
\begin{equation}\label{eqn:smoother}
\calA_0(v,v)\lesssim \mathcal{R}(v,v)\quad\mbox{and}\quad
 \mathcal{R}(v,v) \simeq h^{-2} \|v\|^{2}_{0,\kappa}\;, \quad \forall\, v\in V^{CR}_h\;.
\end{equation}
\end{lemma}
\begin{proof}
We only need to show this inequality for Jacobi smoother, since the Jacobi  and the
symmetric Gauss-Seidel methods are equivalent for any SPD matrix, see for example \cite[Proposition~6.12]{VassilevskiP-2008aa} or \cite[Lemma~3.3]{Zikatanov.L2008}. 

For any $v\in \Vcr$, we write $v =\sum_{e\in \Eho} v_e\varphi^{CR}_e$ where $\varphi^{CR}_e$ is the basis function with respect to $e\in \Eho$. Note that for Jacobi smoother, we  have 
$$
	\mathcal{R}(v, v) = \sum_{e\in \Eho} v_{e}^{2} \calA_{0}(\varphi^{CR}_e, \varphi^{CR}_e).
$$
For any $e\in \Eho$, let $\calE(e):=\{e'\in \Eho: \,\,\, e'\subset \partial\K, \quad \K\in \Th \quad  \partial\K\supset e\, \}$. 
Then, Cauchy-Schwarz and the arithmetic-geometric inequalities give 
\begin{eqnarray*}
\calA_0(v,v) &=& \sum_{e\in \Eho}\sum_{e^\prime\in \calE(e)}
\calA_0(\varphi_e^{CR},\varphi_{e^\prime}^{CR})v_e v_{e^{\prime}} \\
&\leq &  \sum_{e\in \Eho}\sum_{e^\prime\in \calE(e)}
\sqrt{\calA_0(\varphi_e^{CR},\varphi_e^{CR})}
\sqrt{\calA_0(\varphi_{e^\prime}^{CR},\varphi_{e^\prime}^{CR})}v_e v_{e^{\prime}} \\
&\leq & 
\frac12\sum_{e\in \Eho}\sum_{e^\prime\in \calE(e)}\left[
\calA_0(\varphi_e^{CR},\varphi_e^{CR}) v_e^2
+\calA_0(\varphi_{e^\prime}^{CR},\varphi_{e^\prime}^{CR})v_{e^{\prime}}^2\right]\\
&= & 
\sum_{e\in \Eho}
\calA_0(\varphi_e^{CR},\varphi_e^{CR}) v_e^2\leq c_s
\sum_{e\in \Eho}\calA_0(\varphi_e^{CR},\varphi_e^{CR}) v_e^2 = c_s \calR(v,v).
\end{eqnarray*}
The constant $c_s$ above only depends on the cardinality  $\#\calE(e)$, which is bounded by $5$ in 2D and $7$ in 3D.  This proves the first inequality in \eqref{eqn:smoother}.

Since the mesh is quasi-uniform, 
for any $v=\sum_{e} v_e\varphi_e^{CR}\in \Vcr$ and $\K\in \Th$, we have
\begin{equation}
\label{eqn:equiv-discrete-norm}
\|v\|_{0,\kappa,\K}^2 \simeq \sum_{e\subset \partial \K} v_e^2 \|\varphi_e^{CR}\|_{0,\kappa,\K}^2.
\end{equation}
Now by direct calculation, for any basis function $\varphi_{e}^{CR}$ we have
\begin{equation}\label{eq:basis-grad-l2}
h^{-2}\|\varphi_{e}^{CR}\|_{0,\kappa,\K}^{2}\simeq 
\|\nabla \varphi_{e}^{CR}\|_{0,\kappa,\K}^{2}.
\end{equation}
Therefore, by the equivalence relations
\eqref{eqn:equiv-discrete-norm} and \eqref{eq:basis-grad-l2},
we get
\begin{eqnarray*}
\calR(v,v) &=& 
\sum_{e\in \Eho} v_e^2 \|\nabla \varphi_{e}^{CR}\|_{0,\kappa}^{2}
 =  \sum_{e\in \Eho} v_e^2 \|\nabla \varphi_{e}^{CR}\|_{0,\kappa,\K^{+}\cup \K^{-}}^{2}\\
& = & \sum_{\K\in \Th}  \sum_{e\subset \partial\K} v_e^2\|\nabla \varphi_{e}^{CR}\|_{0,\kappa,\K}^{2}
\simeq  \sum_{\K\in \Th}  \sum_{e\subset \partial\K}
h^{-2}v_e^2\|\varphi_{e}^{CR}\|_{0,\kappa,\K}^{2}\\
& \simeq & h^{-2}\sum_{\K\in \Th}\|v\|_{0,\kappa,\K}^{2} = 
h^{-2}\|v\|_{0,\kappa}^{2}\;,
\end{eqnarray*}
which concludes the proof.
\end{proof}
\subsection{A stable Decomposition}\label{sec:stab-decomp}

In this subsection we give a detailed discussion of the stable decomposition.
The main tool is an 
operator $P_{h}^{\h}: V^{CR}_{h} \to \Vc$ that satisfies certain approximation and stability properties,
as stated in the next Lemma.
\begin{lemma}
\label{lm:interpolation}
There exists an interpolation operator $P_{h}^{\h}: V^{CR}_{h} \to \Vc$
that satisfies the following approximation and stability properties:
\begin{align}
\mbox{ Approximation:  } \qquad \|(I-P_{h}^{\h})v\|_{0,\kappa} & \leq C_a \h|\log 2\h/h|^{1/2} \|v\|_{1,h,\kappa},\quad &\forall\, v\in V^{CR}_h, &&\label{eqn:app}\\
\mbox{ Stability:  } \qquad\qquad\qquad\qquad |P_{h}^{\h}v|_{1,\kappa} & \leq C_s |\log 2\h/h|^{1/2} \|v\|_{1,h,\kappa}\quad &\forall\, v\in V^{CR}_h, &&\label{eqn:stab}
\end{align}
with constants $C_a$ and $C_s$ independent of the coefficient $\kappa$ and mesh size. 
\end{lemma}
A construction of such an operator $P_{h}^{\h}$, and proof of the above results,
are given in the Appendix~\ref{sec:projection}. We would like to
point out that the operator $P_{h}^{\h}$ is not used in the actual
implementation of the preconditioner $B$, as it is plainly
 seen from~\eqref{eq:2Loperator}. However, the operator $P_{h}^{\h}$ and its
  approximation and stability properties play a crucial role in the
  analysis.

Observe that on the right hand side
of \eqref{eqn:app} and \eqref{eqn:stab}, the bounds are given in terms
of the weighted full $H^{1}$-norm $\|v\|_{1,h,\kappa}.$ In general,
one cannot replace the norm $\|v\|_{1,h,\kappa}$ by the energy norm $|v|_{1,h,\kappa}$ induced by the bilinear form $a_{h}(\cdot, \cdot)$. To replace the full norm by the semi-norm, one might use the Poincar\'e-Friedrichs inequality for the nonconforming finite element
space (cf. \cite{Dolejsi.V;Feistauer.M;Felcman.J1999,BrennerS-2003aa}) to get:
\begin{eqnarray*}
	\|v\|^{2}_{0,\kappa} &\le& \left(\max_{T\in \Th} \kappa_{T}\right)  \int_{\Omega}|v|^{2} dx \lesssim  \left(\max_{T\in \Th} \kappa_{T}\right) |v|^{2}_{1,h} \lesssim \frac{\max_{T\in \Th} \kappa_{T}}{\min_{T\in \Th} \kappa_{T}} |v|^{2}_{1, h, \kappa}.
\end{eqnarray*}
From the above inequality, we have: 
\begin{corollary}
\label{cor:inter-jmp}
There exists an interpolation operator $P_{h}^{\h}: V_{h}^{CR} \to \Vc$
satisfying the following approximation and stability properties:
\begin{eqnarray}
		\|(I-P_{h}^{\h}) v \|_{0,\kappa} & \lesssim& \mathcal{J}^{1/2}(\kappa) \h |\log (2\h /h)|^{1/2} |v |_{1,h,\kappa}\;, \quad \forall\, 	v\, \in V_{h}^{CR}\;,\label{eqn:int-jmp-app}\\
	| P_{h}^{\h} v  |_{1,\kappa} &\lesssim& \mathcal{J}^{{1/2}}(\kappa) |\log (2\h /h)|^{1/2} |v |_{1,h,\kappa}\;, \quad \forall\, 	v\, \in V_{h}^{CR}\;,\label{eqn:int-jmp-stab}
\end{eqnarray}	
where $\mathcal{J}(\kappa) = \max_{T\in \Th} \kappa_{T} / \min_{T\in \Th} \kappa_{T}$ is the jump of the coefficient.
\end{corollary}
The approximation and stability properties given in
Corollary~\ref{cor:inter-jmp} depend on the coefficient variation
$\mathcal{J}(\kappa)$. However, by imposing some 
constraints on the finite element space $V_{h}^{CR}$, it is possible to get rid of this dependence obtaining a robust result.
Following~\cite[Definition
4.1]{Toselli.A;Widlund.O2005} we introduce
the index set of \emph{floating} subdomains (the subdomains not touching the
Dirichlet boundary):
\begin{equation}\label{def:I}
 {\mathcal I}:=\left\{\,\, i \,\, :\,\, \mbox{meas}_{d-1} (\partial \Omega
\cap \partial \Omega_{i}) = 0\,\right\}\;.
\end{equation} 
We then introduce the subspace
$\widetilde{V}_{h}^{CR} \subset V_{h}^{CR}$:
\begin{equation}\label{vtilde}
  \widetilde{V}_{h}^{CR}:= 
\left\{v \in V_{h}^{CR}: \int_{\Omega_{i}} v dx =0 \;\; \forall i\in
  {\mathcal I} \right\}.
\end{equation}
The key feature of the above subspace is the fact that the
Poincar\'e-Friedrichs inequality for nonconforming finite elements
space (cf. \cite{Dolejsi.V;Feistauer.M;Felcman.J1999,BrennerS-2003aa}) now holds
on each subdomain, which allows us to replace the full norm
$\|v\|_{1,h,\kappa}$ by the semi-norm $|v|_{1,h,\kappa}$, for any
$v\in\widetilde{V}_{h}^{CR}$. 

We remark that the condition on the zero-average in \eqref{vtilde}, is not essential; other conditions could be used (see \cite{Toselli.A;Widlund.O2005}) as
long as they allow for the application of a Poincar\'e-type
inequality. At this point, we would like to emphasize that the dimension of
$\widetilde{V}_{h}^{CR}$ is related to the number of floating
subdomains and in fact: $\dim(\widetilde{V}_{h}^{CR})={\rm dim}(V_{h}^{CR}) -
m_0$, where $m_{0} = \#\mathcal{I}$ is the cardinality of
$\mathcal{I}$.

By restricting now the action of the operator
$P_{h}^{\h}$ to functions in $\widetilde{V}_{h}^{CR}$, we have the
following result, as an easy corollary from
Lemma~\ref{lm:interpolation}. Its proof follows (as mentioned above) by applying Poincar\'e-Friederichs inequality (for nonconforming) on each subdomain.
\begin{corollary}\label{cor:P}
  Let $\widetilde{V}_{h}^{CR}\subset V_h^{CR}$ be the subspace defined in
  \eqref{vtilde}. Then, there exist an operator $P_{h}^{\h}:
  V_{h}^{CR} \to \Vc$ satisfying
\begin{eqnarray}
		\|(I-P_{h}^{\h}) v \|_{0,\kappa} & \lesssim&  \h |\log (2\h /h)|^{1/2} |v |_{1,h,\kappa}\;, \quad \forall\, 	v\, \in V_{h}^{CR}\;,\label{eqn:wapp}\\
	| P_{h}^{\h} v  |_{1,\kappa} &\lesssim&  |\log (2\h /h)|^{1/2} |v |_{1,h,\kappa}\;, \quad \forall\, 	v\, \in V_{h}^{CR}\;.\label{eqn:wstab}
\end{eqnarray}	
\end{corollary}

With the aid of the results from~Corollary~\ref{cor:inter-jmp} and
Corollary~\ref{cor:P}, we can finally show the stability of the decomposition \eqref{eqn:space-decomp}.
\begin{lemma}
	\label{lm:stab-decomp}
	For any $v \in V_{h}^{CR}$, let $\chi=P^{\h}_h(v)\, \in\Vc$, then the following stable decomposition property holds:
	\begin{equation}
		\label{eqn:stab-decomp0}
		\mathcal{R}(v-\chi, v-\chi)  + a(\chi,\chi) \lesssim \mathcal{J}(\kappa) (\h/h)^{2} |\log 2\h/h| |v|^{2}_{1,h,\kappa} \;.
	\end{equation}
In particular, for any $v \in \widetilde{V}_{h}^{CR}$ we have 
	\begin{equation}
		\label{eqn:stab-decomp}
		\mathcal{R}(v-\chi, v-\chi)  + a(\chi,\chi) \lesssim (\h/h)^{2}|\log 2\h/h| |v|^{2}_{1,h,\kappa} \;.
	\end{equation}
\end{lemma}
\begin{proof}
  Below, we give a proof \eqref{eqn:stab-decomp}. Given any $v\in \widetilde{V}_{h}^{CR}$, let $\chi \in \Vc$ be
  defined as $\chi:=P^{\h}_h v $. By the
  approximation property \eqref{eqn:wapp} of $P^{\h}_h$ given in
Corollary~\ref{cor:P}, we have
\begin{align*}
\mathcal{R}(v- \chi, v- \chi)&\lesssim h^{-2} \| v-\chi\|^{2}_{0,\kappa} 
= h^{-2} \|v-P^{\h}_{h}v\|^{2}_{0,\kappa}
\lesssim (\h/h)^{2} |\log 2\h/h||v|_{1,h,\kappa}^{2},
\end{align*}
where in the first inequality, we have used \eqref{eqn:smoother} from Lemma~\ref{lm:smoothCR}.
For the second term, the stability \eqref{eqn:wstab} of $P^{\h}_h$ from Corollary \ref{cor:P} gives,
\begin{equation*}
a(\chi,\chi)=|P^{\h}_h v |_{1,\kappa}^{2}\lesssim |\log 2\h /h| |v
|^{2}_{1,h,\kappa}.
\end{equation*}
The proof of \eqref{eqn:stab-decomp} is complete. The proof of \eqref{eqn:stab-decomp0} is essentially the same but using Corollary~\ref{cor:inter-jmp} instead of  Corollary~\ref{cor:P}.
\end{proof}

We have now all ingredients to complete the proof of Theorem~\ref{teo1}.
\begin{proof} [Proof of Theorem \ref{teo1}.]
  To estimate the maximum eigenvalue of $BA_{0}^{CR}$,
  let $\chi \in \Vc$ and $v \in V_{h}^{CR}$ be arbitrary. We set
  $v_0=(v-\chi)$, and so $v = v_0+\chi$.  The Cauchy-Schwarz
  inequality and Lemma~\ref{lm:smoothCR}  yield
\begin{align*}
	\calA_{0} (v, v)&= \calA_{0} ( v_{0} +  \chi,  v_{0} +  \chi)  \le 2(\calA_0(v_{0}, v_{0}) +  \calA_{0} (  \chi,   \chi)) \le c_{1}\left( \mathcal{R}(v_0,v_0)+ a(\chi,\chi)\right), &&
	\end{align*}
where $c_{1} = 2\max\{c_{s}, 1\}$, with $c_{s}$ (defined in the proof of  Lemma~\ref{lm:smoothCR}), is a constant independent of $\kappa$ and mesh size. 
 Using the identity \eqref{eq:additive-inverse} and the fact that $\chi\in \Vc$ is arbitrary, we have
 \[
	\calA_{0} (v, v)\le c_{1}( B^{-1} v ,  v), \quad \forall v\in V_{h}^{CR} .
\]
Hence, 
$$\lambda_{\max}(BA_{0}^{CR}) = \max_{v\in V_{h}^{CR}} \frac{\calA_{0}(v, v)}{(B^{-1}v, v)} =\max_{v\in V_{h}^{CR}} \frac{(B^{-1} BA_{0}^{CR} v, v)}{(B^{-1}v, v)} \le c_{1},$$
which is uniformly bounded, independently of the coefficient and the mesh size.  

Let $\varpi=\h/h$ be the ratio of the mesh sizes. For the lower bounds of $\lambda_{\min}$ and $\lambda_{m_{0} + 1}$, Lemma~\ref{lm:stab-decomp} with $\chi = P_{h}^{\h} v$ together with  \eqref{eq:additive-inverse} give
\begin{eqnarray*}
	&(B^{-1} v, v) \le \mathcal{R}(v-\chi, v-\chi)  + \left|\chi \right|_{1,\kappa}^{2} \lesssim \mathcal{J}(\kappa) \varpi^{2} |\log 2\varpi | \calA_0(v,v),  &\forall v \in \Vcr,\\
	&(B^{-1} v, v) \le \mathcal{R}(v-\chi, v-\chi)  + \left|\chi \right|_{1,\kappa}^{2} \lesssim \varpi^{2} |\log 2\varpi |\calA_0(v,v), &\forall v\in \widetilde{V}_{h}^{CR}.
\end{eqnarray*}
The first inequality implies that 
$$\lambda_{\min} (BA_{0}^{CR})  = \min_{v\in V_{h}^{CR}} \frac{\calA_{0} (v,v)}{(B^{-1}v, v)}\gtrsim \frac{1}{  \mathcal{J}(\kappa) \varpi^{2} |\log 2\varpi |}.$$ 

The second inequality, together with the fact that $\dim(\widetilde{V}_{h}^{CR}) = \dim(V_{h}^{CR}) - m_{0}$ and the \emph{minimax principle}~\cite[Theorem 8.1.2]{1996GolubG_Van-LoanC-aa}) gives
$$\lambda_{m_{0}+1}(BA_{0}^{CR}) \ge \min_{v\in \widetilde{V}_{h}^{CR}} \frac{\calA_{0} (v,v)}{ (B^{-1}v,v)} \gtrsim \frac{1}{\varpi^{2} |\log 2\varpi |}.$$ 
Therefore, the condition number $\mathcal{K}(BA_{0}^{CR})$ and the effective
condition $ \mathcal{K}_{m_{0}}(BA_{0}^{CR})$ can be respectively bounded by
\begin{equation*}
	\mathcal{K}(BA_{0}^{CR}) \le C_{0} \mathcal{J}(\kappa)\varpi^{2} |\log 2\varpi |,  \mbox{ and } \mathcal{K}_{m_{0}}(BA_{0}^{CR}) \leq C_{1} \varpi^{2}|\log 2\varpi |,
\end{equation*}
 with $C_{0} $ and $C_{1}$, constants independent of the coefficient and mesh size.  
The inequality \eqref{eqn:2LCG-rate} then follows directly from \eqref{eqn:CG}.
\end{proof}

\subsection{Multilevel Preconditioner for $\calA_0(\cdot,\cdot)$ on
  $V_{h}^{CR}$
\label{sec:multilevel}}
We now introduce a multilevel preconditioner, using the two-level theory
developed before. The idea is to replace $[A^{C}]^{-1}$ in~\eqref{eq:2Loperator} with a spectrally
equivalent operator $B^C:\Vc\mapsto \Vc$, which corresponds to the additive
BPX preconditioner 
(see e.g.~\cite{1990BrambleJ_PasciakJ_XuJ-aa,XuJ-1992aa}).  

Given a sequence of quasi-uniform triangulations $\cT_{j}$ for $j =0, 1, \cdots, J$, we denote by $W_{j}=V_{h_{j}}^{\rm{conf}}\;\; (j=0, 1,\cdots, J) $ and consider the family of nested conforming spaces (defined w.r.t. the family of partitions  $\{\cT_{j}\}_{j=0}^{J}$): 
$$W_{0} \subset W_{1} \subset \cdots \subset W_{J}\;.$$
 Here, we assume that the coarsest triangulation $\cT_{0}$ resolves the jump in the coefficient, and without loss of generality, we also assume that  $h_{j} \simeq 2^{-j} \;\;(j = 0, \cdots, J)$ and $h = h_{J}$. The space decomposition that we use to define the multilevel BPX preconditioner is:
\begin{equation}\label{space:decom2}
	V_h^{CR}= V_{h}^{CR} + \sum_{j=0}^{J} W_{j} = \sum_{j=0}^{J+1} W_{j},
\end{equation}
where we have denoted $W_{J+1}=\Vcr$. 
For $j=0, \cdots, J$ we denote  by $A_{j}^{C}$ 
the operator corresponding to  the restriction of $a(\cdot,\cdot)$ to $W_j$, namely 
\[
(A_j^C v_j,w_j) = a(v_j,w_j), \quad \forall v_j\in W_j, \quad \forall
w_j\in W_j.
\] 
The operator form of the
multilevel preconditioner then reads:
\begin{equation}\label{def-Bm}
B_{\textrm{ML}}: \Vcr \mapsto \Vcr, \qquad\quad
B_{\textrm{ML}} := [A^C_0]^{-1}Q^C_0 + \sum_{j=1}^{J+1} R_j^{-1}Q_j.
\end{equation}
Here, $Q_j:\Vcr\mapsto W_j$ is the $L_2$-orthogonal projection on
$W_j$ for $j=0,\ldots,J$ and we set $Q_{J+1}=I$. We use an exact solver on
the coarsest grid. With this notation in hand, one can prove that 
\begin{equation}\label{def-bilm}
(B_{\textrm{ML}}^{-1}v,v)=
\inf_{\sum_{j=0}^{J+1} w_j=v}\left[a(w_0,w_0) + \sum_{j=1}^{J+1} \mathcal{R}_j(w_j,w_j)\right].
\end{equation}
Here $\calR_j(\cdot,\cdot)$, $j=1,\ldots,(J+1)$ correspond to Jacobi or symmetric Gauss-Seidel smoothers,
and the proof of~\eqref{def-bilm} is similar to \eqref{eq:additive-inverse} for the two-level case. 

Next two results will be used in our convergence analysis.
\begin{lemma}[{\cite[Lemma 4.2]{XuJZhuY-2008aa}}]\label{yyy}
Let $\mathcal{R}_{j}(\cdot, \cdot)$ be the Jacobi or the symmetric Gauss-Siedel smoother for the solution of the discretization \eqref{a:conf} on $W_j$ space ($\forall j = 1, \cdots, J$). Then,
\begin{equation*}
a(w,w)\lesssim \mathcal{R}_{j}(w,w)\lesssim h_{j}^{-2} \|w\|^{2}_{0,\kappa} \quad \forall\, w\in W_j.
\end{equation*}
\end{lemma}

We also need the following strengthened Cauchy Schwarz inequality. 
\begin{lemma}[{Strengthened Cauchy Schwarz, cf. \cite[Lemma 6.2]{XuJ-1992aa}}]
	\label{lm:scs}
	For $j =1, \cdots, J-1$ and $j< l\le J$,
    there exists a constant $\gamma\in (0,1)$ such that
	\begin{equation}
	\label{eqn:scs}
		a(w_{l}, w_{j}) \lesssim \gamma^{l-j} (h_{l}^{-1}\|w_{l}\|_{0,\kappa}) (h_{j}^{-1}\|w_{j}\|_{0,\kappa}), \quad \forall w_{l}\in W_{l}, w_{j}\in W_{j}.
	\end{equation}
\end{lemma}

The main result of this section is the following:
\begin{theorem}\label{teo2}
  Let $B_{\rm ML}$ be the multilevel preconditioner defined in \eqref{def-Bm}.  Then, the condition number $\mathcal{K}(B_{\rm ML}A_{0}^{CR})$ satisfies:
  	\begin{equation*}
		\mathcal{K}(B_{\rm ML}A_{0}^{CR}) \le C_{0} \mathcal{J}(\kappa)J^{2}\;,
	\end{equation*}
		where $J$ is the number of levels, and $\mathcal{J}(\kappa):= \max_{T\in \Th} \kappa_{T} /\min_{T\in \Th} \kappa_{T}$ is the jump of the coefficient. Moreover, there exists an integer $m_{0}$ depending only on the distribution of the coefficient $\kappa$ such that the $m_{0}$-th effective
  condition number $\mathcal{K}_{m_{0}}(B_{\rm ML} A_0^{CR})$ satisfies:
	\begin{equation*}
	 \mathcal{K}_{m_{0}}(B_{\rm ML} A_0^{CR}) \le C_{1} J^{2}\;,
	\end{equation*}
	where the constants $C_{0}, C_{1}>0$ are independent of the coefficients and mesh size. Hence, the convergence rate of the PCG algorithm can be bounded as
	\begin{equation}
\label{eqn:CG-rate}
\frac{|u-u_k |_{1,h,\kappa}}{|u-u_0 |_{1,h,\kappa}}\le 2(C_{0}\mathcal{J}(\kappa)J^{2}-1)^{m_0}
\left(\frac{\sqrt{C_{1}}J-1}{\sqrt{C_{1}}J+1}\right)^{k-{m_0}} \;.
\end{equation}
\end{theorem}

\begin{proof}
We first give a bound on $\lambda_{\max}(B_{\textrm{ML}}A_0^{CR})$.
Let $v\in W_{J+1}=V^{CR}_h$ be arbitrary,
and let $\{w_j\}_{j=0}^{J+1}$ be any decomposition of $v$, namely
$v = \sum_{j=0}^{J+1} w_{j}$, with $w_{j}\in W_{j}$.
By the Cauchy-Schwarz inequality, we have
\begin{eqnarray*}
  \calA_{0} (v,v) &=& \calA_{0}\left(\sum_{j=0}^{J+1} w_{j},
    \sum_{j=0}^{J+1} w_{j}\right) \le 
3\left(a(w_{0}, w_{0})  
  + \sum_{i=1}^{J}\sum_{j=1}^{J} a\left( w_{i},  w_{j}\right) 
  + \calA_{0}(w_{J+1}, w_{J+1}) \right).
\end{eqnarray*}
By Lemma~\ref{yyy}, the strengthened Cauchy-Schwarz inequality (Lemma \ref{lm:scs}),  and the smoothing property of $\calR_{J+1}(\cdot, \cdot)$~\eqref{eqn:smoother}, we get:
\begin{eqnarray*}
  \calA_{0} (v,v) &\lesssim &  a(w_{0}, w_{0}) 
  + \sum_{i=1}^{J}\sum_{j=1}^{J} \gamma^{|i-j|} \left(h_{i}^{-1}\|w_{i}\|_{0,\kappa}\right) \left(h_{j}^{-1} \|w_{j}\|_{0,\kappa}\right)
  + \calR_{J+1}(w_{J+1}, w_{J+1})\\
&\lesssim & \left( a(w_{0}, w_{0}) +  \sum_{j=1}^{J} \mathcal{R}_{j} (w_{j}, w_{j}) + \calR(w_{J+1}, w_{J+1}) \right),
\end{eqnarray*}
where in the second inequality, we used the fact that the spectral 
radius of the matrix $(\gamma^{|i-j|})_{J\times J}$ is uniformly bounded by
$(1-\gamma)^{-1}$. 
Since the decomposition of $v$ was arbitrary, taking the infimum above over all
such decompositions and using  the identity~\eqref{def-bilm} then gives
\begin{equation*}
\calA_0(v,v) \lesssim (B_{\textrm{ML}}^{-1}v,v), \quad \forall v\in \Vcr,
\end{equation*}
which shows that $\lambda_{\max}(B_{\textrm{ML}}A_{0}^{CR}) \lesssim 1$.

 Similar to the proof of Theorem~\ref{teo1}, the estimates on the lower bound for $\lambda_{\min}$ and $\lambda_{m_{0}}$ rely on the stability of the decomposition.  For this purpose,  we make use of the interpolation operator and its properties introduced in \S\ref{sec:stab-decomp}.
To simplify the notation, we set 
$P_{j}:=P_{h}^{h_j}:
  \widetilde{V}_h^{CR}\longrightarrow W_j$, for $j=0,\ldots ,J$,
  and set $P_{J+1}=I$ and $P_{-1}=0$. 
Given any $v\in V_{h}^{CR}$, we define the decomposition of $v$ as
\begin{equation*}
	v =P_{J+1}v= \sum_{j=0}^{J+1} (P_{j} - P_{j-1}) v=\sum_{j=0}^{J+1} w_{j}, 
\quad\mbox{where} 
\quad w_j=(P_{j} - P_{j-1}) v.
\end{equation*}
Clearly, $w_j\in W_j$ for $j =1, \cdots, (J+1)$ and $w_{0} = P_{0}
v\in W_0$.  Triangle inequality and the smoothing properties of $\calR_{j} \;(j=1, \cdots, J+1)$ from Lemma~\ref{yyy}  and Lemma~\ref{lm:smoothCR}, give 
\begin{eqnarray}
	a(w_0,w_0) + \sum_{j=1}^{J+1} \mathcal{R}_{j}(w_{j}, w_{j}) 
&\lesssim&  |P_{0} v|_{1, \kappa}^{2} + \sum_{j=1} ^{J+1}h_{j}^{-2} \|(P_{j} - P_{j-1})v\|^{2}_{0,\kappa}\nonumber\\
	&\lesssim & | P_{0} v|_{1,\kappa}^{2} + \sum_{j=0}^{J} h_{j}^{-2}\|v - P_{j}v\|_{0,\kappa}^{2}. \label{eqn:aaa} 
\end{eqnarray}
Using in \eqref{eqn:aaa},  the approximation property \eqref{eqn:int-jmp-app}  of $P_{j}\; (j=0, \cdots, J)$ and the stability property \eqref{eqn:int-jmp-stab} of $P_{0}$ given in Corollary~\ref{cor:inter-jmp}, we obtain
\begin{eqnarray*}
	(B_{\rm ML}^{-1} v, v) &\le& a(w_0,w_0) + \sum_{j=1}^{J+1} \mathcal{R}_{j}(w_{j}, w_{j}) \\
	&\lesssim&  \mathcal{J}(\kappa)\left(\sum_{j=0}^{J}   |\log{h_{j}}|\right) |v|_{1, h, \kappa}^{2}  \lesssim\mathcal{J}(\kappa) J^{2} \calA_{0}(v, v),  \qquad \forall v\in \Vcr.
\end{eqnarray*}
This gives the estimate on minimal eigenvalue of $B_{\rm ML} A_{0}^{CR}$ as $$ \lambda_{\min}(B_{\rm ML}A_{0}^{CR}) \gtrsim 1/\left(\mathcal{J}(\kappa) J^{2}\right). $$
Similarly, if we use  Corollary~\ref{cor:P} in \eqref{eqn:aaa}, we obtain
$$
	(B_{\rm ML}^{-1} v, v) \lesssim J^{2} \calA_{0}(v, v),  \qquad \forall v\in \widetilde{V}_{h}^{CR}.
$$
Therefore, $\lambda_{m_{0}+1}(B_{\textrm{ML}}^{-1} A_0^{CR}) \gtrsim 1/J^{2}$ by the \emph{minimax principle} and the result follows. 
\end{proof}
\begin{remark}
  Similar results hold also for the multiplicative multilevel methods such as the $V$-cycle.
   These results can be derived from estimates comparing multiplicative and additive preconditioners  given in~\cite[Theorem~4]{GriebelM_OswaldP-1995aa}
or~\cite[Theorem~4.2]{2008ChoD_XuJZikatanovL-aa}. We refer to \cite{Zhu.Y2011} for a detailed analysis and numerical justification. 
\end{remark}


\section{Numerical Experiments}
\label{sec:numerical}
We consider the model problem \eqref{eqn:model} in the square $\Omega
= [-1, 1]^{2} $ with coefficients:
$$
	\kappa(x) = \left\{\begin{array}{ll}
		1.0, &\forall x\in [-0.5, 0]^{2} \cup [0, 0.5]^{2},\\
		\epsilon, &\text{elsewhere}.
		\end{array}\right.
$$

In all of the following experiments, $\epsilon$ varies from $10^{-5}$ up
to $10^{5}$, covering a wide range of variations of the coefficients. 
The set of experiments is carried out on a family of structured triangulations; we consider uniform refinement with a structured
initial triangulation on level $0$ with $32$ elements and mesh size
$h=2^{-1}$. 
This initial mesh resolves the jump in the coefficients.
Each refined triangulation is then obtained by subdividing each 
element of the previous level into four congruent elements.
The number of degrees of freedom $N_{\ell}$ in the DG
discretizations on each level satisfies $N_{\ell} = 4^{\ell} N_{0}$
for $\ell = 0, 1, 2, 3, 4$ with $N_{0} = 96$. 
We  consider the IP($\beta$)-0 method \eqref{ipA0} with penalty parameter $\alpha= 8$. 

We use  the basis \eqref{basis-CR}-\eqref{basis-z} for the computations. To solve the resulting linear systems  we use Algorithm \ref{algo0}. Due to the block
structure\eqref{Acal0} of $\mathbb{A}_{0}$ (matrix representation of
$A_{0}$ in the basis \eqref{basis-CR}-\eqref{basis-z}) we only need to
numerically verify the effectiveness of the solvers for each block;
$\mathbb{A}_{0}^{vv}$ and $\mathbb{A}_{0}^{zz}$. Recall that for any
choice of $\theta =0, \pm 1$, the block $\mathbb{A}_{0}^{vv}$ is the
same (since it is the stiffness matrix of the Crouzeix-Raviart
discretization \eqref{prob:cr}), while the block $\mathbb{A}_{0}^{zz}$
is different for different values of $\theta$, but it is always an SPD matrix.
To solve each of these smaller systems we use the  preconditioned CG,  for which we have set the tolerance to TOL=$10^{-7}$ for the stopping criteria based on the residual; namely, if $r_0$ is the initial residual and $r^{k}$ is
the residual at iteration $k$, the PCG iteration process is terminated at iteration $k$ if
$\|r^{k}\|_{\ell_2}/\|r^{0}\|_{\ell_2} <10^{-7}$.
The experiments were carried out on an IMAC (OS X)
with 2.93 GHz Intel Core i7, and 8 GB 1333 MHz DDR3.

The systems corresponding to $\mathbb{A}_{0}^{zz}$ are solved by a PCG algorithm using its diagonal $\mathbb{D}_{z}$ as a preconditioner.  The estimated condition numbers of $\mathbb{D}_{z}^{-1}\mathbb{A}_0^{zz}$ for  SIPG($\beta$)-0   are reported in Table~\ref{tab:SIPG-Z}.
\begin{table}[ht]
	\centering
\small
\begin{tabular}{c|c||c|c|c|c|c|c|c}\hline\hline
 &  & \multicolumn{5}{c}{$\epsilon$}\\ \cline{3-9}
levels & h &$10^{-5}$ & $10^{-3}$ & $10^{-1}$ & $1$ & $10^{1}$     &$10^{3}$    & $10^{5}$
\\ \hline\hline
0 & $2^{-1}$&  1.73 (14)&  1.73 (12)&  1.73  &  1.73 (9)  &  1.72 (10)  &  1.73 (12) &1.73 (13)
\\ \hline
1 & $2^{-2}$&  1.72 (15)&   1.72 (13)&   1.72&  1.72 (10)&  1.72 (10) &  1.72 (12)&  1.72 (14)
\\ \hline
2 & $2^{-3}$&  1.72 (15)&   1.72 (13)&   1.72&  1.71 (10)&   1.7 (10) &  1.71 (12) &  1.72 (15)
\\ \hline
3 & $2^{-4}$&  1.72 (15)&    1.72 (12)&  1.71&  1.71 (10)&  1.69 (10) &  1.69 (12) &  1.69 (16)
\\ \hline\hline
 \end{tabular}

%
%

\caption{Estimated condition numbers $\mathcal{K}(\mathbb{D}_{z}^{-1}\mathbb{A}_0^{zz})$ (number of PCG iterations) for the block $\mathbb{A}_{0}^{zz}$ in SIPG($\beta$)-0 discretization. \label{tab:SIPG-Z}}
\end{table}
Observe that the condition numbers of $\mathbb{D}_{z}^{-1}\mathbb{A}_0^{zz}$ are uniformly bounded and close to 1, which confirms  the result established in Lemma \ref{le:diag0}; i.e., that $\mathbb{A}_0^{zz}$ is spectrally equivalent to its diagonal. Similar results, although not reported here, were found for the  NIPG($\beta$)-0 and IIPG($\beta$)-0 methods.  The system $\mathbb{A}_{0}^{vv}$ arising from the restriction of
$\calA_0(\cdot,\cdot)$ to the Crouzeix-Raviart space is solved by a
PCG algorithm with the two-level preconditioners defined in
\eqref{eq:2Loperator}, for which we use 5 symmetric Gauss-Seidel iterations as smoother.  

Figure~\ref{fig:SIPGe-6lvl5} shows the spectrum of the preconditioned system for $\epsilon=10^{-5}$ and the mesh size $h=2^{-5}$. In this example, we have taken $\h = h$, so $\calT_{\h} = \Th$. 
\begin{figure}[htbp]
    \begin{center}
         \includegraphics[width=0.55\textwidth]{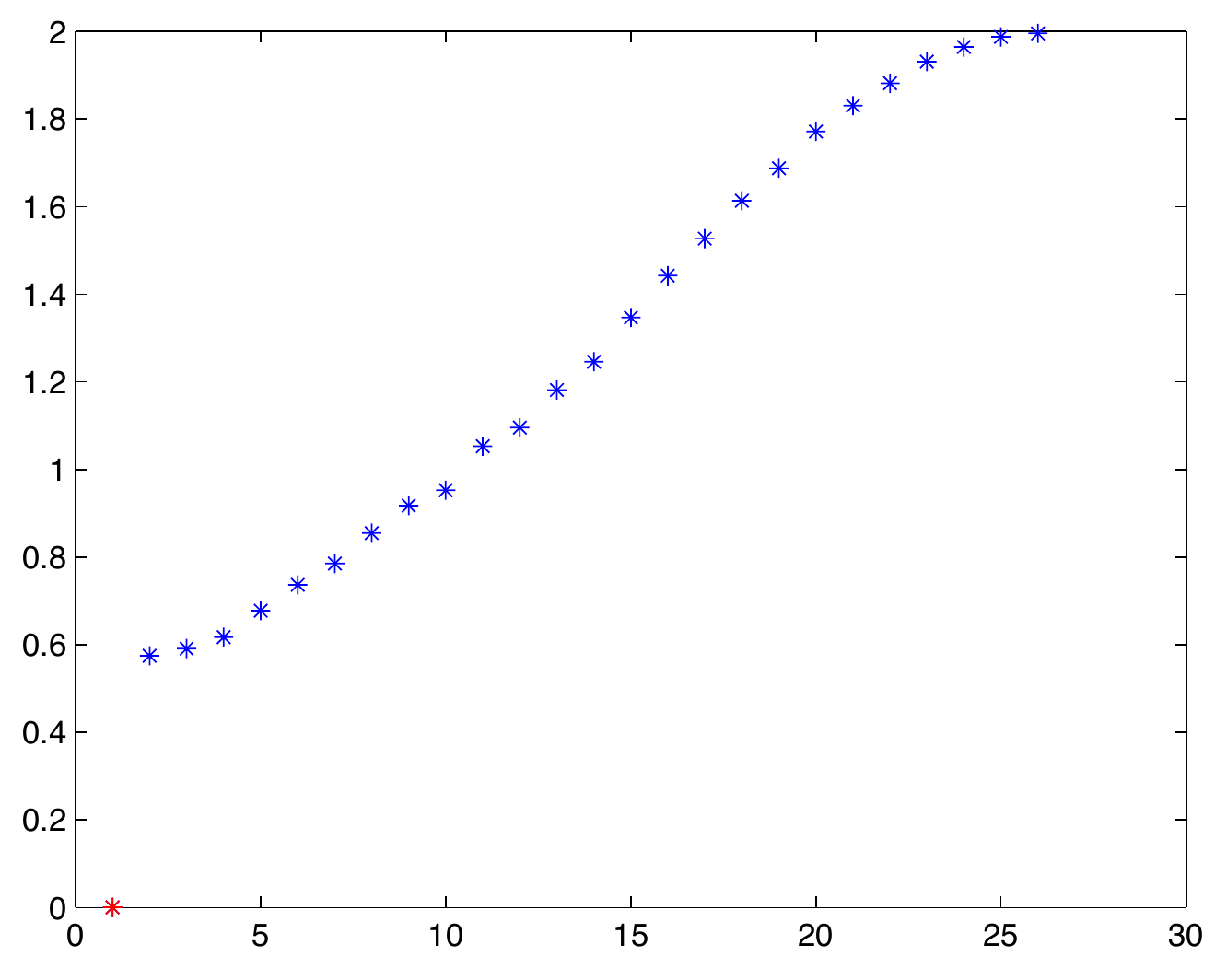}
    \end{center}
    \caption{Eigenvalue distribution of $\mathbb{B}\mathbb{A}_0^{vv}$ for $\epsilon = 10^{-5}$ and $h=2^{-5}$. }    \label{fig:SIPGe-6lvl5}
 \end{figure}  
Note that there is only one (very small) eigenvalue close to zero (which may be related to the fact that there are only 2 different values for the coefficients). In Table
\ref{tab:2L0} we report the estimated condition number
$\mathcal{K}(\mathbb{B}\mathbb{A}_{0}^{vv})$ and the effective
condition number (denoted by
$\mathcal{K}_1(\mathbb{B}\mathbb{A}_{0}^{vv})$). Observe that the
estimated condition number
$\mathcal{K}(\mathbb{B}\mathbb{A}_{0}^{vv})$ deteriorates with respect
to the magnitude of the jump in the coefficient. In contrast, the effective condition number
$\mathcal{K}_{1}(\mathbb{B}\mathbb{A}_{0}^{vv})$ is uniformly bounded with
respect to both the mesh size and the jump of coefficient, as
predicted by Theorem~\ref{teo1}. 
\begin{table}[ht]

\begin{tabular}{c|c||c|c|c|c|c}
\hline\hline
\multirow{2}{*}{$\epsilon$}
 & levels &  0         & 1        & 2        & 3        & 4 \\\cline{2-7}
&  $h$   & $2^{-1}$  & $2^{-2}$   & $2^{-3}$  & $2^{-4}$  & $2^{-5}$\\
\hline\hline
\multirow{2}{*}{$10^{-5}$}
 & $\mathcal{K}(\mathbb{B}\mathbb{A}_{0}^{vv})$ & 3e+4 (12)& 3.31e+4 (19)& 2.77e+4 (22)& 2.37e+4 (21)& 2.08e+4 (21)\\
 & $\mathcal{K}_{1}(\mathbb{B}\mathbb{A}_{0}^{vv})$ &  4.52 &  3.37 &  2.95 &  2.78 &  2.71 \\\hline
\multirow{2}{*}{$10^{-3}$}
 & $\mathcal{K}(\mathbb{B}\mathbb{A}_{0}^{vv})$ &   301 (11)&   333 (15)&   280 (18)&   240 (18)&   211 (18)\\
 & $\mathcal{K}_{1}(\mathbb{B}\mathbb{A}_{0}^{vv})$ &  4.48 &  3.36 &  2.95 &  2.77 &  2.71 \\\hline
\multirow{2}{*}{$10^{-1}$}
 & $\mathcal{K}(\mathbb{B}\mathbb{A}_{0}^{vv})$ &  4.42 (10)&  5.22 (13)&  4.91 (14)&   4.7 (14)&  4.59 (14)\\
 & $\mathcal{K}_{1}(\mathbb{B}\mathbb{A}_{0}^{vv})$ &  2.97 &  2.89 &  2.69 &   2.6 &  2.57 \\\hline
\multirow{2}{*}{$1$}
 & $\mathcal{K}(\mathbb{B}\mathbb{A}_{0}^{vv})$ &  2.16 (8)&  2.25 (11)&  2.29 (12)&   2.3 (12)&  2.33 (12)\\
 & $\mathcal{K}_{1}(\mathbb{B}\mathbb{A}_{0}^{vv})$ &  2.06 &  2.16 &  2.21 &  2.19 &  2.18 \\\hline
\multirow{2}{*}{$10^{1}$}
 & $\mathcal{K}(\mathbb{B}\mathbb{A}_{0}^{vv})$ &  2.33 (9)&  3.16 (12)&  3.58 (13)&   3.8 (14)&  3.95 (14)\\
 & $\mathcal{K}_{1}(\mathbb{B}\mathbb{A}_{0}^{vv})$ &   2.3 &  2.63 &  2.66 &  2.62 &  2.61 \\\hline
\multirow{2}{*}{$10^{3}$}
 & $\mathcal{K}(\mathbb{B}\mathbb{A}_{0}^{vv})$ &  2.54 (9)&  4.12 (13)&  5.37 (14)&  6.56 (15)&  7.79 (16)\\
 & $\mathcal{K}_{1}(\mathbb{B}\mathbb{A}_{0}^{vv})$ &   2.4 &  2.82 &  2.85 &   2.8 &  2.78 \\\hline
\multirow{2}{*}{$10^{5}$}
 & $\mathcal{K}(\mathbb{B}\mathbb{A}_{0}^{vv})$ &  2.55 (9)&  4.13 (13)&  5.41 (15)&  6.62 (16)&  7.89 (17)\\
 & $\mathcal{K}_{1}(\mathbb{B}\mathbb{A}_{0}^{vv})$ &   2.4 &  2.83 &  2.85 &   2.8 &  2.78 \\\hline
\hline
\end{tabular}

\caption{Two level preconditioner for $\mathbb{A}_{0}^{vv}$ on $V_{h}^{CR}$ with $\h = h$. \label{tab:2L0}}
\end{table}

For comparison, we also present the results obtained with different choices of coarse grid $\h = 2h, \; 4h$, reported in Tables \ref{tab:2L1} -\ref{tab:2L2}.  As we can see from these two tables, the effective condition number is uniformly bounded with respect to the coefficient and mesh size. However, comparing to the results in Table~\ref{tab:2L0}, it seems that the effective condition numbers get larger when we use a  coarser grid. These observations coincide with the conclusion in Theorem~\ref{teo1}.
\begin{table}[!ht]

\begin{tabular}{c|c||c|c|c|c|c}
\hline\hline
\multirow{2}{*}{$\epsilon$}
 & levels &  0         & 1        & 2        & 3        & 4 \\\cline{2-7}
&  $h$   & $2^{-1}$  & $2^{-2}$   & $2^{-3}$  & $2^{-4}$  & $2^{-5}$\\
\hline\hline
\multirow{2}{*}{$10^{-5}$}
 & $\mathcal{K}(\mathbb{B}\mathbb{A}_{0}^{vv})$ & X& 4.92e+4 (18)& 4.28e+4 (24)& 3.66e+4 (26)& 3.21e+4 (27)\\
 & $\mathcal{K}_{1}(\mathbb{B}\mathbb{A}_{0}^{vv})$ & X&  4.27 &  3.61 &  3.38 &  3.33 \\\hline
\multirow{2}{*}{$10^{-3}$}
 & $\mathcal{K}(\mathbb{B}\mathbb{A}_{0}^{vv})$ & X&   494 (16)&   431 (21)&   370 (21)&   325 (21)\\
 & $\mathcal{K}_{1}(\mathbb{B}\mathbb{A}_{0}^{vv})$ & X&  4.26 &  3.61 &  3.38 &  3.34 \\\hline
\multirow{2}{*}{$10^{-1}$}
 & $\mathcal{K}(\mathbb{B}\mathbb{A}_{0}^{vv})$ & X&  7.14 (14)&  6.69 (16)&  6.35 (16)&  6.19 (16)\\
 & $\mathcal{K}_{1}(\mathbb{B}\mathbb{A}_{0}^{vv})$ & X&  3.46 &  3.27 &   3.2 &  3.19 \\\hline
\multirow{2}{*}{$1$}
 & $\mathcal{K}(\mathbb{B}\mathbb{A}_{0}^{vv})$ & X&  2.63 (11)&  2.75 (13)&  2.91 (14)&  2.97 (14)\\
 & $\mathcal{K}_{1}(\mathbb{B}\mathbb{A}_{0}^{vv})$ & X&  2.32 &  2.61 &  2.63 &  2.61 \\\hline
\multirow{2}{*}{$10^{1}$}
 & $\mathcal{K}(\mathbb{B}\mathbb{A}_{0}^{vv})$ & X&  3.74 (13)&   4.3 (15)&  4.48 (16)&  4.67 (16)\\
 & $\mathcal{K}_{1}(\mathbb{B}\mathbb{A}_{0}^{vv})$ & X&  3.33 &  3.38 &  3.32 &  3.29 \\\hline
\multirow{2}{*}{$10^{3}$}
 & $\mathcal{K}(\mathbb{B}\mathbb{A}_{0}^{vv})$ & X&  4.93 (14)&  6.59 (16)&  8.02 (18)&  9.55 (18)\\
 & $\mathcal{K}_{1}(\mathbb{B}\mathbb{A}_{0}^{vv})$ & X&  3.64 &  3.65 &  3.56 &  3.49 \\\hline
\multirow{2}{*}{$10^{5}$}
 & $\mathcal{K}(\mathbb{B}\mathbb{A}_{0}^{vv})$ & X&  4.95 (14)&  6.63 (16)&  8.02 (18)&  9.66 (19)\\
 & $\mathcal{K}_{1}(\mathbb{B}\mathbb{A}_{0}^{vv})$ & X&  3.65 &  3.65 &  3.53 &  3.49 \\\hline
\hline
\end{tabular}

	\caption{Two level preconditioner for $\mathbb{A}_{0}^{vv}$ on
          $V_{h}^{CR}$ with $\h = 2h$.\label{tab:2L1}}

\end{table}
\begin{table}[!ht]

\begin{tabular}{c|c||c|c|c|c|c}
\hline\hline
\multirow{2}{*}{$\epsilon$}
 & levels &  0         & 1        & 2        & 3        & 4 \\\cline{2-7}
&  $h$   & $2^{-1}$  & $2^{-2}$   & $2^{-3}$  & $2^{-4}$  & $2^{-5}$\\
\hline\hline
\multirow{2}{*}{$10^{-5}$}
 & $\mathcal{K}(\mathbb{B}\mathbb{A}_{0}^{vv})$ & X& X& 7.89e+4 (31)& 7.29e+4 (34)& 6.41e+4 (35)\\
 & $\mathcal{K}_{1}(\mathbb{B}\mathbb{A}_{0}^{vv})$ & X& X&  6.58 &  5.99 &  5.97 \\\hline
\multirow{2}{*}{$10^{-3}$}
 & $\mathcal{K}(\mathbb{B}\mathbb{A}_{0}^{vv})$ & X& X&   793 (25)&   733 (28)&   646 (29)\\
 & $\mathcal{K}_{1}(\mathbb{B}\mathbb{A}_{0}^{vv})$ & X& X&  6.57 &  5.99 &  5.97 \\\hline
\multirow{2}{*}{$10^{-1}$}
 & $\mathcal{K}(\mathbb{B}\mathbb{A}_{0}^{vv})$ & X& X&  12.2 (20)&  11.6 (22)&  11.4 (22)\\
 & $\mathcal{K}_{1}(\mathbb{B}\mathbb{A}_{0}^{vv})$ & X& X&  5.58 &  5.69 &  5.76 \\\hline
\multirow{2}{*}{$1$}
 & $\mathcal{K}(\mathbb{B}\mathbb{A}_{0}^{vv})$ & X& X&  4.73 (17)&  5.22 (19)&  5.32 (19)\\
 & $\mathcal{K}_{1}(\mathbb{B}\mathbb{A}_{0}^{vv})$ & X& X&  3.99 &  4.75 &   4.8 \\\hline
\multirow{2}{*}{$10^{1}$}
 & $\mathcal{K}(\mathbb{B}\mathbb{A}_{0}^{vv})$ & X& X&  7.55 (19)&  6.84 (21)&  6.97 (22)\\
 & $\mathcal{K}_{1}(\mathbb{B}\mathbb{A}_{0}^{vv})$ & X& X&  6.34 &  5.63 &  5.95 \\\hline
\multirow{2}{*}{$10^{3}$}
 & $\mathcal{K}(\mathbb{B}\mathbb{A}_{0}^{vv})$ & X& X&  11.2 (20)&  12.2 (23)&  14.6 (25)\\
 & $\mathcal{K}_{1}(\mathbb{B}\mathbb{A}_{0}^{vv})$ & X& X&  6.99 &  6.11 &  6.39 \\\hline
\multirow{2}{*}{$10^{5}$}
 & $\mathcal{K}(\mathbb{B}\mathbb{A}_{0}^{vv})$ & X& X&  11.3 (20)&  12.3 (23)&  14.9 (26)\\
 & $\mathcal{K}_{1}(\mathbb{B}\mathbb{A}_{0}^{vv})$ & X& X&     7 &  6.12 &   6.4 \\\hline
\hline
\end{tabular}

	\caption{Two level preconditioner for $\mathbb{A}_{0}^{vv}$ on
          $V_{h}^{CR}$ with $\h = 4h$.\label{tab:2L2}}
\end{table}

We now present the results corresponding to the multilevel preconditioners as defined in \eqref{def-Bm}. In Table \ref{tab:bpx16} we report the estimated condition number $\mathcal{K}(\mathbb{B}\mathbb{A}_{0}^{vv})$ and the effective condition number (denoted by $\mathcal{K}_1(\mathbb{B}\mathbb{A}_{0}^{vv})$) for the BPX.
Also for the BPX, we use 5 symmetric Gauss-Siedel iterations as a smoother.
Observe that the estimated condition number
$\mathcal{K}(\mathbb{B}\mathbb{A}_{0}^{vv})$ deteriorates with respect
to the magnitude of the jump in coefficient. On the other hand, the effective condition number
$\mathcal{K}_{1}(\mathbb{B}\mathbb{A}_{0}^{vv})$ is nearly uniformly bounded with
respect to both the mesh size and the jump of the coefficient, as
predicted by Theorem \ref{teo2}.
\begin{table}[!ht]
\centering
\small
\begin{tabular}{c|c||c|c|c|c|c}
\hline\hline
\multirow{2}{*}{$\epsilon$}
 & levels &  0         & 1        & 2        & 3        & 4 \\
\cline{2-7}
&  $h$   & $2^{-1}$  & $2^{-2}$   & $2^{-3}$  & $2^{-4}$  & $2^{-5}$\\
\hline\hline
\multirow{2}{*}{$10^{-5}$}
 & $\mathcal{K}(\mathbb{B}\mathbb{A}_{0}^{vv})$ & 3e+4 (12)& 5.03e+4 (27)& 6.77e+4 (33)& 8.64e+4 (37)& 1.06e+5 (42)\\
 & $\mathcal{K}_{1}(\mathbb{B}\mathbb{A}_{0}^{vv})$ &  4.52 &  5.69 &  6.81 &   7.9 &  9.03 \\\hline
\multirow{2}{*}{$10^{-3}$}
 & $\mathcal{K}(\mathbb{B}\mathbb{A}_{0}^{vv})$ &   301 (11)&   506 (22)&   680 (27)&   868 (31)& 1.06e+03 (35)\\
 & $\mathcal{K}_{1}(\mathbb{B}\mathbb{A}_{0}^{vv})$ &  4.49 &  5.65 &  6.78 &  7.86 &  8.98 \\\hline
\multirow{2}{*}{$10^{-1}$}
 & $\mathcal{K}(\mathbb{B}\mathbb{A}_{0}^{vv})$ &  4.42 (10)&   7.5 (16)&  9.92 (20)&  12.5 (24)&  15.1 (26)\\
 & $\mathcal{K}_{1}(\mathbb{B}\mathbb{A}_{0}^{vv})$ &  2.97 &  4.22 &  5.28 &   6.3 &  7.41 \\\hline
\multirow{2}{*}{$1$}
 & $\mathcal{K}(\mathbb{B}\mathbb{A}_{0}^{vv})$ &  2.16 (8)&  3.32 (13)&  4.45 (17)&  5.61 (20)&  6.67 (22)\\
 & $\mathcal{K}_{1}(\mathbb{B}\mathbb{A}_{0}^{vv})$ &  2.07 &  3.17 &  4.25 &  5.23 &  6.24 \\\hline
\multirow{2}{*}{$10^{1}$}
 & $\mathcal{K}(\mathbb{B}\mathbb{A}_{0}^{vv})$ &  2.33 (9)&  4.58 (14)&  6.69 (19)&  8.75 (22)&    11 (26)\\
 & $\mathcal{K}_{1}(\mathbb{B}\mathbb{A}_{0}^{vv})$ &   2.3 &  3.84 &  5.06 &  6.19 &  7.31 \\\hline
\multirow{2}{*}{$10^{3}$}
 & $\mathcal{K}(\mathbb{B}\mathbb{A}_{0}^{vv})$ &  2.54 (9)&  5.92 (16)&  10.1 (21)&  15.6 (25)&    23 (29)\\
 & $\mathcal{K}_{1}(\mathbb{B}\mathbb{A}_{0}^{vv})$ &   2.4 &  4.11 &  5.42 &  6.62 &  7.81 \\\hline
\multirow{2}{*}{$10^{5}$}
 & $\mathcal{K}(\mathbb{B}\mathbb{A}_{0}^{vv})$ &  2.55 (9)&  5.94 (16)&  10.2 (21)&  15.7 (25)&  23.3 (29)\\
 & $\mathcal{K}_{1}(\mathbb{B}\mathbb{A}_{0}^{vv})$ &   2.4 &  4.11 &  5.43 &  6.62 &  7.81 \\\hline
\hline
\end{tabular}

	\caption{PCG with BPX (additive) preconditioner for solving on
          $V_{h}^{CR}$.\label{tab:bpx16}}
\end{table}
Moreover, we also observe that the effective condition numbers grow linearly with respect to the number of levels, which is better than the quadratic growth in Theorem~\ref{teo2}. This issue will be further investigated in the future.


\section{Solvers for IP($\beta$)-1
  Methods}\label{subsect:preconditioners-for-type-0 and-1}
  
We now briefly discuss how the  preconditioners developed here for the  IP($\beta$)-0  can be used or extended for preconditioning the IP($\beta$)-1 methods \eqref{ipA}. We follow~\cite{Ayuso-de-DiosBZikatanovL-2009aa}. 

\subsection{Solvers for the SIPG($\beta$)-1 method\label{subsec:solverSIPG}}
From the spectral equivalence
given in Lemma \ref{lm:equivA:A0}, it follows that any of the
preconditioners designed for $\mathcal{A}_0(\cdot,\cdot)$ result in
an efficient solver for $\mathcal{A}(\cdot,\cdot)$. Motivated by the block diagonal form of $\mathbb{A}_0$ (cf. \eqref{Acal0}), we use the decomposition \eqref{splitting0} and define the following
block-Jacobi preconditioner:
\begin{equation}\label{eqn:B1DG}
\mbox{Block-Jacobi: $B_1^{DG}:= [R^{z}]^{-1} + \widetilde{B}Q^{CR}\;,$}
\end{equation}
where $R^{z}$ denotes the operator corresponding to the diagonal of $\calA(\cdot,\cdot)$ restricted to $\calZ_{\beta}$ and  $\widetilde{B}$ refers to the corresponding
multilevel preconditioner for the symmetric SIPG($\beta$)-1 method (i.e.,
including the jump-jump term). The next result is a simple consequence of Theorem~\ref{teo2} (focusing only on the asymptotic result) together with
Lemma \ref{lm:equivA:A0}. 
 \begin{theorem}\label{teo3}
 Let $B^{DG}$ be the preconditioner defined through \eqref{eqn:B1DG}. Let $m_0$ be the number of floating subdomains. Then, the following estimate holds for the effective condition $\mathcal{K}_{m_{0}}(B^{DG}A)$:
	\begin{equation*}
		\mathcal{K}_{m_{0}}(B^{DG}A) \le C J^{2}\;.
	\end{equation*}
	The constant $C>0$ above is independent of
        the variation in the coefficients and mesh size. \end{theorem}

In Table \ref{tab:DAS3-SIPG} are given the estimated condition numbers 
of $\mathcal{K}(\mathbb{B}_{1}^{DG}\mathbb{A})$  together with the estimated effective condition numbers
$\mathcal{K}_1(\mathbb{B}_{1}^{DG}\mathbb{A})$, and the number of PCG
iterations required for convergence. 
\begin{table}[ht]
	\centering
\small
\begin{tabular}{c|c||c|c|c|c}
\hline\hline
\multirow{2}{*}{$\epsilon$}
& levels &  0               & 1               & 2               & 3\\  \cline{2-6}
&  $h$   & $2^{-1}$   & $2^{-2}$   & $2^{-3}$  & $2^{-4}$\\
\hline\hline
\multirow{2}{*}{$10^{-5}$}
 & $\mathcal{K}(\mathbb{B}_{1}^{DG}\mathbb{A})$ & 2.85e+4 (44)& 3.37e+4 (44)& 3.1e+4 (46)& 2.85e+4 (46)\\
 & $\mathcal{K}_{1}(\mathbb{B}_{1}^{DG}\mathbb{A})$ &  6.27 &  6.33 &  6.45 &  6.49 \\\hline
\multirow{2}{*}{$10^{-3}$}
 & $\mathcal{K}(\mathbb{B}_{1}^{DG}\mathbb{A})$ &   288 (33)&   340 (34)&   313 (34)&   289 (32)\\
 & $\mathcal{K}_{1}(\mathbb{B}_{1}^{DG}\mathbb{A})$ &  6.24 &   6.3 &  6.42 &  6.46 \\\hline
\multirow{2}{*}{$10^{-1}$}
 & $\mathcal{K}(\mathbb{B}_{1}^{DG}\mathbb{A})$ &  7.25 (22)&  7.33 (22)&  7.21 (22)&  7.13 (22)\\
 & $\mathcal{K}_{1}(\mathbb{B}_{1}^{DG}\mathbb{A})$ &  5.62 &   5.6 &  5.71 &  5.73 \\\hline
\multirow{2}{*}{$1$}
 & $\mathcal{K}(\mathbb{B}_{1}^{DG}\mathbb{A})$ &  5.53 (19)&  5.76 (20)&   5.8 (20)&  5.83 (20)\\
 & $\mathcal{K}_{1}(\mathbb{B}_{1}^{DG}\mathbb{A})$ &  5.17 &  5.45 &  5.46 &  5.46 \\\hline
\multirow{2}{*}{$10^{1}$}
 & $\mathcal{K}(\mathbb{B}_{1}^{DG}\mathbb{A})$ &  6.66 (22)&  7.16 (23)&  7.16 (23)&  7.43 (23)\\
 & $\mathcal{K}_{1}(\mathbb{B}_{1}^{DG}\mathbb{A})$ &  5.91 &   6.2 &  6.25 &  6.27 \\\hline
\multirow{2}{*}{$10^{3}$}
 & $\mathcal{K}(\mathbb{B}_{1}^{DG}\mathbb{A})$ &  6.38 (27)&  8.98 (30)&  11.1 (31)&  13.5 (32)\\
 & $\mathcal{K}_{1}(\mathbb{B}_{1}^{DG}\mathbb{A})$ &  5.51 &  6.53 &  6.59 &  6.59 \\\hline
\multirow{2}{*}{$10^{5}$}
 & $\mathcal{K}(\mathbb{B}_{1}^{DG}\mathbb{A})$ &  6.91 (33)&  9.02 (36)&  11.3 (39)&  13.8 (40)\\
 & $\mathcal{K}_{1}(\mathbb{B}_{1}^{DG}\mathbb{A})$ &  6.38 &  6.54 &   6.6 &  6.59 \\\hline
\end{tabular}
	
\caption{Estimated condition number
  $\mathcal{K}(\mathbb{B}_{1}^{DG}\mathbb{A})$ (number of PCG
  iterations) and the effective condition number
  $\mathcal{K}_{1}(\mathbb{B}_{1}^{DG}\mathbb{A})$. \label{tab:DAS3-SIPG}}
		
\end{table} 
As can be seen from these two
tables, $\mathcal{K}(\mathbb{B}_{1}^{DG}\mathbb{A})$  deteriorate
rapidly when $\epsilon$ becomes smaller, but 
$\mathcal{K}_1(\mathbb{B}_{1}^{DG}\mathbb{A})$ are nearly uniformly bounded with respect
to the coefficients and mesh size. These results confirm the theory predicted by Theorem~\ref{teo3}.

      \subsection{Solvers for the non-symmetric IIPG($\beta$)-1 and NIPG($\beta$)-1
        methods\label{sec:nonsym-type-1}}
      We consider the following linear iteration:
\begin{algorithm}\label{algo1} 
Given initial guess $u_0$, for $k=0,1\ldots$ until convergence:
\begin{enumerate}
\item[1.] Set $ e_k = B^{DG}(f -A u_k)$;
\item[2.] Update $u_{k+1}=u_k+e_k\;.$
\end{enumerate}
\end{algorithm}
Here, $A: V_{h}^{DG}\mapsto V_{h}^{DG}$ is the operator associated with the bilinear form of either the  NIPG($\beta$)-1 or IIPG($\beta$)-1 methods (\eqref{ipA} with $\theta=1$ and $\theta=0$, respectively):
\begin{equation}\label{eq:def-A}
(A v, w) : = \calA(v,w), \quad \forall v, w\in V_{h}^{DG}.
\end{equation}
Following ~\cite{Ayuso-de-DiosBZikatanovL-2009aa} we consider as preconditioner $B^{DG}$ the symmetric part of $A$, defined by:
\begin{equation}\label{eqn:invAS}
B^{DG}:= A_S^{-1}, \quad \mbox{where}\qquad
 (A_S v,w):=\frac{1}{2}[\calA(v,w)+\calA(w,v)], \quad
\forall v\in V_{h}^{DG}, \quad \forall w\in V_{h}^{DG}.
\end{equation}
We note that from this definition and~\eqref{eq:coer}, we immediately
have that $A_S$ is symmetric and positive definite. 
 The next result guarantees uniform convergence of the linear
iteration in Algorithm~\ref{algo1} with preconditioner $B^{DG}$ given
by~\eqref{eqn:invAS}. The proof follows \cite[Theorem~5.1]{Ayuso-de-DiosBZikatanovL-2009aa} and it is omitted.
\begin{theorem}\label{thm:IIPG}
  Let $\alpha^{\ast}$ be a fixed value of the penalty parameter
  for which the IIPG($\beta$)-0 bilinear form~\eqref{ipA0}
  $\calA_0(\cdot,\cdot)$ is coercive. Let $\calA(\cdot,\cdot)$ be the
  bilinear form of the IIPG($\beta$)-1 method~\eqref{ipA} with penalty
  parameter $\alpha \geq 4\alpha^{\ast}$. Let $B^{DG}=A_S^{-1}$ be the
  iterator in the linear iteration~\ref{algo1}, and let $u_k$ and
  $u_{k+1}$ be two consecutive iterates obtained via this
  algorithm. Then there exists a positive constant $\Lambda<1$ such
  that
\begin{equation}\label{rollo}
\triplenorm{u-u_{k+1}}_{DG}\leq \Lambda \triplenorm{u-u_{k}}_{DG}\;.
\end{equation}
 \end{theorem}
To verify Theorem \ref{thm:IIPG} we have computed the
$\calA$-norm (which is obviously equivalent in $V^{DG}_h$ to the $\triplenorm{\cdot}_{DG}$) of the error propagation operator: $E=I-B^{DG}A=I-A^{-1}_SA$, for different meshes and values of $\epsilon$. This norm gives us the contraction number of the linear
  iteration in Algorithm~\ref{algo1}, and so an estimate for the constant $\Lambda$ in Theorem \ref{thm:IIPG}. The results are reported in Table~\ref{tab:E-IIPG-AS-k1-k2}. 
\begin{table}[!ht]
\centering
\label{tab:E-IIPG-AS-k2}
\begin{tabular}{c|c||c|c|c|c|c|c|c|c|c|c|c}\hline\hline
 &  & \multicolumn{11}{c}{$\epsilon$}\\ \cline{3-13}
levels & h &$10^{-5}$ & $10^{-4}$ & $10^{-3}$ & $10^{-2}$ & $10^{-1}$& $1$ & $10^{1}$& $10^{2}$ & $10^{3}$ & $10^{4}$ & $10^{5}$
\\\hline\hline
0 & $2^{-1}$&  0.20  &   0.20  &   0.20  &   0.20  &   0.20  &   0.20  &   0.19  &   0.19  &   0.19  &   0.19  &   0.19\\\hline
1 & $2^{-2}$&  0.14  &   0.14  &   0.14  &   0.14  &   0.14  &   0.14  &   0.14  &   0.14  &   0.14  &   0.14  &   0.14\\\hline
2 & $2^{-3}$&  0.16  &   0.16  &   0.16  &   0.16  &   0.16  &   0.15  &   0.15  &   0.16  &   0.16  &   0.16  &   0.16\\\hline
3 & $2^{-4}$& 0.16  &    0.16  &   0.16  &   0.16  &  0.16   &  0.16   & 0.16   & 0.16   & 0.16  & 0.16  & 0.16  \\\hline
\hline
\end{tabular}

\caption{
\label{tab:E-IIPG-AS-k1-k2}
Norm of the error propagator $E=(I-A_S^{-1}A)$ for $A$ corresponding to 
IIPG discretization with $\alpha  = 4\alpha^{*}$.
}
\end{table}
 
 More experiments for the IP($\beta$)-1 methods can be found in \cite{ahzz:tech}.

\section*{Acknowledgement}
The work of the first author was partially supported by the Spanish MEC under projects MTM2011-27739-C04-04 and HI2008-0173. The work of the second author was supported in part
  by NSF DMS-0715146, NSF DMS-0915220, and DTRA Award HDTRA-09-1-0036. The work of the third author was supported in part
  by NSF DMS-0715146 and DTRA Award HDTRA-09-1-0036. The work of the fourth author was supported in part
  by NSF DMS-0810982 and NSF OCI-0749202. We would also like to thank the anonymous referees for their carefully proofreading this manuscript. Their suggestions helped a lot to improve this paper.
\appendix

\appendix
\section{Construction of an Interpolation Operator}
\label{sec:projection}
We now construct an interpolation operator which satisfies the
approximation and stability properties \eqref{eqn:app}-\eqref{eqn:stab} in Lemma~\ref{lm:interpolation}.  
To begin with, let us introduce some notation. Given a conforming triangulation $\Th$, recall that $\Eh$ is the set
of edges/faces of $\Th$. Let $S_{h}\subset H^{1}(\Omega)$ be the
conforming $\mathbb{P}^{d}(\Th)\cap\mathcal{C}^{0}(\Omega)$ Lagrange finite element space (quadratics in $d=2$ and cubics for $d=3$). We split the set of DOFs of $S_{h}$ into two subsets $\mathcal{C}(\Th)$ and ${\mathcal N}(\Th)$, where  $\mathcal{C}(\Th)$ contains the DOFs of $V_{h}^{CR}$ corresponding to the barycenters of the edges/faces in $\Eh$  and ${\mathcal N}(\Th)$ contains all the remaining DOFs of $S_{h}$.  We also denote the restriction of $\Th$, $\Eh$, $V_{h}^{CR}$ or $S_{h}$ on a given subdomain $G$ by $\Th(G)$, $\Eh(G)$, $V_{h}^{CR}(G)$ or $S_{h}(G)$, respectively.

Let $\cT_{\h}$ be a coarser mesh, i.e., $\Th$ is either the same as
$\cT_{\h}$ or a refinement of it with $h\le \h.$ We now start building
the operator $P_{h}^{\h}: V_{h}^{CR} \to V_{\h}^{{\rm conf}}$, where we recall that 
$V_{\h}^{{\rm conf}}$ is the piecewise $\mathbb{P}^{1}$ conforming
finite element space defined on $\cT_{\h}.$ The basic idea is to, in each subdomain $\Omega_i$,
embed $V_{h}^{CR}(\Omega_i)$ into $S_{h}(\Omega_i)$. Then we interpolate the result in $V_{\h}^{{\rm conf}}$ on $\cT_{\h}$ using a quasi-interpolation operator.

To embed $V_{h}^{CR}$ into $S_{h}$, we modify the inclusion operator
introduced in \cite{BrennerS-2003aa}, and define it at the subdomain
level as follows. For any $v\in V_{h}^{CR}$ we define $E_{i} :
V_{h}^{CR}(\Omega_{i})\to S_{h}(\Omega_{i})$ on each subdomain
$\Omega_{i}$ as:
\begin{equation}\label{def:ei}
	(E_{i} v)(p)= \left\{
		\begin{array}{ll}
			v(p), & \mbox{ if } p\in \mathcal{C}\left(\Th\right) \cap \overline{\O}_i\\
	\frac{1}{\#M_{p}} \sum_{\K\in M_{p}} v_{T}(p),  & \mbox{ if }  p\in \mathcal{N}\left(\Th\right) \cap \O_i\\
		\frac{1}{\#M^{\partial}_{p}} \sum_{e\in M^{\partial}_{p}} v_{e}(p), & \mbox{ if }  p\in \mathcal{N}\left(\Th\right) \cap \partial\O_i,
		\end{array}\right.
\end{equation}	
where $M_{p} :=\{ T\in \Th (\Omega_{i}): p\in \partial T\}$ is the set
of elements sharing $p$ and $M^{\partial}_{p}:=\{e\in \Eh(\Omega_{i}):
e\subset \partial \Omega_{i} \; \mbox{ s.t.} \;\; p\in \partial e\} $ is the
set of edges on $\partial \O_{i}$ sharing the DOF $p$. Here $\#M_{p}$ and
$\#M^{\partial}_{p}$ are the cardinality of these sets
respectively, and $v_{T}$, $v_{e}$ are the restriction of $v$ on $T$
and $e$ respectively.

Observe that, this construction differs from the one in \cite[Equation (3.1)]{BrennerS-2003aa} in the treatment of the DOFs on $\partial \Omega_{i}$.  From
\eqref{def:ei}, for each DOF $p\in \partial \O_{i}$,
$(E_{i} v)(p)$ contains only the contributions of $v$ from the
boundary of $\O_{i}$, not from the interior.  Therefore, it is obvious that 
$(E_{i} v)(p) \equiv (E_{j} v)(p)$ for any DOF $p$ in the
interior of the interface $\Gamma=\partial\O_i\cap \partial\O_j\,\,\, (i\ne j)$ between the subdomains $\Omega_{i}$ and $\Omega_{j}$. 
This special treatment at boundary points guarantees that the global function
$\eta|_{\Omega_{i}}:=E_{i} v$ is continuous for the points in the interior of each interface. However, this global function $\eta$ will generally be multi-valued at the points on $\partial \Gamma$.  

Although the construction of $E_{i}$ in \eqref{def:ei} is different from \cite{BrennerS-2003aa}, the same analysis in \cite{BrennerS-2003aa} can be carried out here. We summarize the the properties of $E_{i}$ below, and omit the detailed proof.
\begin{lemma}\label{lm:local-iso}
The linear operator $E_{i} : V_{h}^{CR}(\Omega_{i})\to S_{h}(\Omega_{i})$ defined in \eqref{def:ei} satisfies that
$$
	\left|E_{i} v\right|_{1,\Omega_{i}} \simeq |v|_{1,h,\Omega_{i}},\quad \mbox{ and } \quad\left\|v - E_{i} v\right\|_{0,\Omega_{i}} \lesssim h\left|v\right|_{1,h,\Omega_{i}}, \quad \forall v \in V_{h}^{CR}.
$$
\end{lemma}

Let $\mathcal{Q}_i:H^1(\Omega_i)\to
V_{\h}^{\rm{conf}}(\O_i)$  and $\mathcal{Q}_{\Gamma}: H^{1}(\Gamma) \to
V_{\h}^{\rm{conf}}(\Gamma)$ be the Scott-Zhang quasi-interpolation operators
on $\O_i$ and  on the interface $\Gamma \subset \O_{i}$, respectively. We now recall the definition and main properties of these
operators.  In the sequel, we should denote a generic
vertex of $\cT_{\h}$ by $p$. Let $\omega_{p} :=\bigcup \left\{\K\in \cT_{\h}(\Omega_{i}):
p\in \partial \K \right\} \subset \Omega_{i}$ be the local patch containing $p$, and $\omega_{\K} := \bigcup \left\{\omega_{p}:
p\in \partial \K \right\}$ for each $\K \in \cT_{\h}(\Omega_{i})$. Similarly, on the
interface $\Gamma$, we define $\mathcal{O}_{p} :=\bigcup \{e\in
\mathcal{E}_{\h}: e\subset \Gamma \mbox{ and } p\in \partial e \} \subset \Gamma$ and $\mathcal{O}_{e} := \bigcup
\{\mathcal{O}_{p}: p\in \partial e\}$ for each $e \in
\mathcal{E}_{\h}(\Gamma)$. For any vertex $p$, let $\phi_{p}\in V_{\h}^{{\rm
    conf}}$ be the nodal basis function, and define the dual basis
$\theta_{p} \in V_{\h}^{{\rm conf}}(\omega_{p})$ such that
$$
	\int_{\omega_{p}} \theta_{p} v dx = v(p),\qquad \forall v\in V_{\h}^{{\rm conf}}.
$$
To define $\mathcal{Q}_{i}$, let us choose some\footnote{Note that the
  choice of $\K$ may not be unique.} $\K\subset \omega_p$ for each
vertex $p\in \cT_{\h}(\Omega_{i})$. Then, the Scott-Zhang
quasi-interpolation operator is defined by
\begin{equation*}
\mathcal{Q}_i \eta=\sum_{p\in \mathcal{T}_{\h}(\O_i)}\left(\int_{\K}\theta_{p} \eta  dx\right )\phi_p\;, \quad \eta \in H^{1}(\Omega_{i}).
\end{equation*}
The operator $\mathcal{Q}_{\Gamma}$ is defined similarly, but
restricted on the interface $\Gamma$. Both operators enjoy the
following approximation and stability properties (see
\cite{Oswald.P1994,Scott.R;Zhang.S1990} for a proof):
\begin{lemma}\label{lm:int}
  For any $\eta\in H^{1}(\Omega_i)$, the operator
  $\mathcal{Q}_i:H^{1}(\Omega_i) \to V_{\h}^{\rm{conf}}(\Omega_i)$  satisfies:
 \begin{equation}\label{eqn:Qi}
    \|\mathcal{Q}_i \eta\|_{0, \K}\lesssim \|\eta\|_{0,\omega_{\K}}, \;\;  \|\mathcal{Q}_i \eta\|_{1,\K}\lesssim \|\eta\|_{1,\omega_{\K}}\;,\;\; \; \| (I-\mathcal{Q}_i) \eta\|_{0, \K} \lesssim \h \|\eta\|_{1,\omega_{\K}}, \;\; \forall \K\in \cT_{\h}(\Omega_{i}).
	\end{equation}  
For any $\xi\in H^{1}(\Gamma)$, the operator $\mathcal{Q}_{\Gamma}: H^{1}(\Gamma) \to V_{\h}^{\rm{conf}}(\Gamma)$ satisfies the following
  properties:
	\begin{equation}\label{eqn:Qg}
	\|\mathcal{Q}_{\Gamma} \xi\|_{0,e} \lesssim \|\xi\|_{0, \mathcal{O}_{e}} 
	\;, \quad \|(I-\mathcal{Q}_{\Gamma}) \xi)\|_{0,e} \lesssim \h\|\xi\|_{1,\mathcal{O}_e} \quad \forall e\in \mathcal{E}_{\h} (\Gamma)\;.
	\end{equation}
	Furthermore, both operators are linear preserving; i.e. $\mathcal{Q}_i\eta \equiv\eta $ for any $\eta\in V_{\h}^{\rm{conf}} (\O_{i})$, and similarly $\mathcal{Q}_{\Gamma}\xi \equiv\xi$ for any $\xi\in
        V_{\h}^{\rm{conf}}(\Gamma)$.
\end{lemma}

Now we are ready to define the interpolation operator $P_{h}^{\h}: V_{h}^{CR} \to
V^{\rm{conf}}_{\h}$:
\begin{equation}
\label{eq:loc-int}
	\left(P_{h}^{\h} v\right)|_{\Omega_{i}} (p)=\left\{ 
		\begin{array}{ll}
			\left(\mathcal{Q}_{i} E_{i} v\right) (p)	, & \mbox{ if } p\in \Omega_{i}\\
			\left(\mathcal{Q}_{\Gamma} E_{i} v \right)(p), & \mbox{ if } p\in {\rm int}(\Gamma) \mbox{ for each side } \Gamma \subset \partial\Omega_{i}\\
			0,  & \mbox{elsewhere}\\
		\end{array}
		\right.,
\end{equation}
where ${\rm int}(\Gamma)$ is the interior of $\Gamma.$ From the definition of $E_{i}$ in \eqref{def:ei}, if $p$ is a vertex of $\calT_{\h}$ in the interior of the interface $\Gamma = \O_{i} \cap \O_{j}$, we have $(E_{i}v)(p) = (E_{j}v)(p),$ which implies  $\left(\mathcal{Q}_{\Gamma} E_{i} v \right)(p) \equiv \left(\mathcal{Q}_{\Gamma} E_{j} v \right)(p)$. The special
treatment for the interface in \eqref{eq:loc-int} guarantees the global
continuity of $P_{h}^{\h} v$. Thus,  $P_{h}^{\h} v \in V^{\rm{conf}}_{\h}$ is well-defined. Now, we
show that the operator $P_{h}^{\h}$ defined in \eqref{eq:loc-int} does satisfy the
approximation and stability properties
\eqref{eqn:app}-\eqref{eqn:stab}:
\begin{lemma}\label{lema:PP}
	For any $v \in V_{h}^{CR}$, the operator $P_{h}^{\h} : V_{h}^{CR } \to V_{\h}^{\rm{conf}}$ satisfies 
	\begin{eqnarray}
		\|(I-P_{h}^{\h}) v \|_{0,\kappa} & \lesssim& \h|\log (2\h/h)|^{1/2} \|v\|_{1,h,\kappa}, \label{approx:0P}\\
		|P_{h}^{\h} v |_{1,\kappa} & \lesssim& |\log (2\h/h)|^{1/2} \|v\|_{1,h,\kappa}\;. \label{eqn:Ph-stab}
	\end{eqnarray}
\end{lemma} 
\begin{proof}
The proof follows the ideas from \cite[Lemma 4.6]{Bramble.J;Xu.J1991}, adapted to the present situation. We start by showing \eqref{approx:0P}. Using triangle
  inequality, Lemma~\ref{lm:local-iso}, together with the approximation
  result \eqref{eqn:Qi} of the $\mathcal{Q}_i$ from
  Lemma~\ref{lm:int}, we have
\begin{align}\label{caspa10}
	\|v - P_{h}^{\h} v\|_{0,\Omega_i} &\le \|v - \mathcal{Q}_{i}E_{i} v\|_{0,\Omega_i} + \|\mathcal{Q}_{i}E_{i} v - P_{h}^{\h} v\|_{0,\Omega_i} && \nonumber\\
	&\le \|v - E_{i} v\|_{0,\Omega_i} + \| (I- \mathcal{Q}_{i}) E_{i} v\|_{0,\Omega_i} +\|\mathcal{Q}_{i}E_{i} v - P_{h}^{\h} v\|_{0,\Omega_i} &&\nonumber \\
	&\lesssim h|v|_{1,h,\O_i}+ \h\|E_i v\|_{1,\O_i} +\|\mathcal{Q}_{i}E_{i} v - P_{h}^{\h} v\|_{0,\Omega_i}\; &&\nonumber \\
	&\lesssim h|v|_{1,h,\O_i}+ \h \| v\|_{1,h,\O_i}  +\|\mathcal{Q}_{i}E_{i} v - P_{h}^{\h} v\|_{0,\Omega_i}\;. &&
\end{align}	
Hence, to show the inequality \eqref{approx:0P} we only need to
estimate $\|\mathcal{Q}_{i}E_{i} v - P_{h}^{\h} v\|_{0,\Omega_i}$.

To simplify the notation, throughout the proof we set $\chi =
{P}_{\h}^{h} v\in V_{\h}^{\rm{conf}}$ as defined in
\eqref{eq:loc-int}, and denote $\chi_{i}:=\mathcal{Q}_{i}
E_{i}v$. From the definition of $P_{h}^{\h}$ in \eqref{eq:loc-int},
$\chi(p)\equiv \chi_{i}(p)$ when $p$ is a vertex of $\calT_{\h}$ in the interior of $\Omega_{i}$, and they are
different only on the boundary vertices. So by using discrete $L^{2}$ norm, we
have
    \begin{align}
    \|\mathcal{Q}_{i}E_{i} v - P_{h}^{\h} v\|_{0,\Omega_{i}}&= \|\chi-\chi_i\|_{0,\Omega_i}^2 
     \lesssim 
      \sum_{\Gamma\subset \partial
      \Omega_i}\sum_{p\in \Gamma} \h^{d}(\chi-\chi_i)^2(p) &&\nonumber\\
      &= \sum_{\Gamma\subset \partial \Omega_i}\left(\sum_{p\in
      {\rm int}(\Gamma)} \h^{d} \left( \mathcal{Q}_{\Gamma} E_{i} v- \chi_i\right)^2(p)+\sum_{p\in \partial \Gamma} \h^{d} \chi_i^2(p)\right) &&\nonumber\\
      &\lesssim \sum_{\Gamma\subset \partial
      \Omega_i}\left(\sum_{e\in \mathcal{E}_{\h}(\Gamma)} \h \|\mathcal{Q}_{\Gamma} E_{i} v- \chi_i\|_{0,e}^2+ \h^{2}\| \chi_i\|_{0, \partial \Gamma }^2\right).\label{caspa} &&
    \end{align}
    Below, we try to bound those two terms appearing in the last expression of \eqref{caspa}.
    
    For the first term in \eqref{caspa}, we observe that $\chi_{i} \equiv \mathcal{Q}_{\Gamma} \chi_{i}$ by Lemma~\ref{lm:int}.  Then by the $L^{2}$-stability property \eqref{eqn:Qg} of $\mathcal{Q}_{\Gamma}$,  we obtain
    \begin{align*}
      \h\|\mathcal{Q}_{\Gamma} E_{i} v - \chi_i\|_{0,e}^2&=  \h\|\mathcal{Q}_{\Gamma} E_{i} v - \mathcal{Q}_{\Gamma}\chi_i\|_{0,e}^2 \lesssim
      \h \| E_{i} v- \chi_i\|_{0, \mathcal{O}_{e} }^2 &&\\
      & \lesssim  \h (\h)^{-1}  \| (I-\mathcal{Q}_{i}) E_{i} v\|_{0, \omega_e}^2+
      \h^{2}|E_{i} v- \mathcal{Q}_{i} E_{i} v|_{1,\omega_e}^2 &&\\
      &\lesssim  \h^{2} \|E_{i}v\|_{1,\omega_e}^2, &&
    \end{align*}
    where in the second inequality, we used the standard trace inequality (cf. \cite[Lemma 2.1]{Bramble.J;Xu.J1991}), and in the last step we used the properties \eqref{eqn:Qi} of $\mathcal{Q}_{i}$.  Here $\omega_{e} := \cup\{T\in \cT_{\h} (\Omega_{i}): \partial T\cap \mathcal{O}_{e} \neq \emptyset\}$. Summing up the above inequality for all edges/faces on $\partial \Omega_{i}$,  we obtain that
    \begin{equation}\label{caspa2}
    \sum_{\Gamma\subset \partial  \Omega_i}\sum_{e\in \mathcal{E}_{\h}(\Gamma)} \h\| \mathcal{Q}_{\Gamma} E_{i}v - \chi_{i} \|_{0,e}^{2} \lesssim \h^{2} \|E_{i}v\|^{2}_{1,\Omega_{i}} \lesssim \h^{2}  \|v\|^{2}_{1,h,\Omega_{i}}.
    \end{equation} 
 
    To bound the second term in \eqref{caspa} we have to distinguish
    between the $2D$ and $3D$ cases. In the $2D$ case, $\Gamma$ is a
    one-dimensional edge of $\Omega_{i}$, so $\partial\Gamma$ reduces to its two endpoints, say
    $\{p,q\}$. Hence,
  \begin{equation*}
  \|\chi_i\|^{2}_{0,\partial\Gamma}=(|\chi_i(p)|^{2}+|\chi_i(q)|^{2})\leq  \|\chi_i\|_{0,\infty,\omega_p }^{2}+  \|\chi_i\|_{0,\infty,\omega_q }^{2}\;, \qquad  \partial\Gamma =\{ p, q\}\;.
  \end{equation*}
  To bound each of the above two terms on the right side, we use the
  two-dimensional discrete Sobolev inequality \cite[Lemma
  2.3]{Bramble.J;Xu.J1991};
  \begin{equation}\label{log:sob2D}
   \|\chi_i\|_{0,\infty, \omega_p}\leq C\left(\log{\frac{{\rm diam}(\omega_p)}{h}}\right)^{1/2}\|\chi_i\|_{1,\omega_p}\;.
   \end{equation}
  So summing over all $\Gamma \subset \partial\O_i$ the resulting
   estimate, we finally get
     \begin{align}
    \sum_{\Gamma\subset \partial \Omega_i}
      \|\chi_i\|^2_{0, \partial \Gamma } &\lesssim   \sum_{\Gamma\subset \partial \Omega_i}\sum_{p\in \partial \Gamma } \log{\left(\frac{{\rm diam}(\omega_p)}{h}\right)}\|\chi_i\|^{2}_{1,\omega_p} \lesssim
\log{\left(\frac{2\h}{h}\right)}\|\chi_i\|^{2}_{1,\O_i} && \nonumber \\
&=\log{\left(\frac{2\h}{h}\right)}\left\|\mathcal{Q}_{i} E_{i} v\right\|_{1,\O_i}^2 \lesssim \log{\left(\frac{2\h}{h}\right)} \left\| v\right\|_{1,h,\O_i}^2\;, \label{caspa5} &&
\end{align}
where in the second inequality we used the fact ${{\rm diam}(\omega_p)}\simeq 2\h$, and in the last step we used the inequality \eqref{eqn:Qi} of $\mathcal{Q}_{i}$ and the properties of $E_{i}$ in Lemma~\ref{lm:local-iso}.

In $3D$,  $\Gamma \subset \partial\O_i$ is a two-dimensional face of $\O_{i}$. So
$\partial\Gamma$ is a union of edges in the triangulation $\{e \in \calE_{\h}: e\subset \partial \Gamma\}$. In this case, we use the following the discrete Sobolev
inequality \cite[Lemma 2.4]{Bramble.J;Xu.J1991} (instead of
\eqref{log:sob2D} in 2D case):
 \begin{equation*}
  \|\chi_i\|^{2}_{0,\partial\Gamma}=\sum_{e\subset \partial \Gamma} \|\chi_i\|^{2}_{0,e} \lesssim \sum_{e\subset \partial \Gamma} \log{\left(\frac{{\rm diam}(\omega_e)}{h}\right)}\|\chi_i\|^{2}_{1,\omega_e}\;.
    \end{equation*}
    Summing the above estimate over all $\Gamma \subset \partial\O_i$ and
    using, as before, the inequality \eqref{eqn:Qi} of $\mathcal{Q}_{i}$ together with the properties of $E_i$ given in Lemma
    \ref{lm:local-iso}, we find
\begin{align}
    \sum_{\Gamma\subset \partial \Omega_i}
      \|\chi_i\|^2_{0, \partial \Gamma } &\lesssim  \sum_{\Gamma\subset \partial \Omega_i}\sum_{e\subset \partial \Gamma} \log{\left(\frac{{\rm diam}(\omega_e)}{h}\right)}\|\chi_i\|^{2}_{1,\omega_e}\lesssim 
\log{\left(\frac{2\h}{h}\right)}\sum_{\Gamma\subset \partial \Omega_i}\sum_{e\subset \partial \Gamma}\|\chi_i\|^{2}_{1,\omega_e}  &&\nonumber\\
&\lesssim 
\log{\left(\frac{2\h}{h}\right)}\|\chi_i\|^{2}_{1,\O_i}\lesssim 
\log{\left(\frac{2\h}{h}\right)} \left\| v\right\|_{1,h,\O_i}^2\;. && \label{eqn:3dg}
\end{align}
Now, substituting \eqref{eqn:3dg} (or \eqref{caspa5} when $d=2$)  and \eqref{caspa2} into \eqref{caspa},  we finally
get
         \begin{equation*}
           \|\mathcal{Q}_{i}E_{i} v - P_{h}^{\h} v\|^{2}_{0,\Omega_i}= \|\chi-\chi_i\|_{0,\Omega_i}^2 \lesssim \h^{2}  \|v\|^{2}_{1,h,\Omega_{i}} +
 \h^{2}\log{\left(\frac{2\h}{h}\right)} \left\| v\right\|_{1,h,\O_i}^2\;.
\end{equation*}
The inequality \eqref{approx:0P} then follows by inserting the above estimate in \eqref{caspa10}.

Finally we show the stability of $P_{h}^{\h}$ \eqref{eqn:Ph-stab}. Note that $P_{h}^{\h} v \in V_{\h}^{\mbox{\rm{conf}}}$ and
        $v\in V^{CR}_{h}$. To deal with possibly different mesh
        sizes we consider the local $L^{2}$-projection
        $\mathcal{P}_{\K} : L^{2}(\K) \lor \mathbb{P}^{1}(\K)$ for any
        $\K\in \mathcal{T}_{\h}$. For $\h >h$, such an element is the
        union of other subelements in the partition $\Th$. Then,
        adding and subtracting $\mathcal{P}_{\K} v$, triangle
        inequality together with inverse inequality and the
        approximation property \eqref{approx:0P}, gives
\begin{align*}
| P_{h}^{\h} v |_{1,\K} &\leq | P_{h}^{\h} v-\mathcal{P}_{\K} v |_{1,\K}+ | \mathcal{P}_{\K} v |_{1,\K} \leq C(\h)^{-1}\| P_{h}^{\h} v -\mathcal{P}_{\K} v \|_{0,\K} + | \mathcal{P}_{\K} v |_{1,\K}&&\\
&\leq C(\h)^{-1}\left( \| P_{h}^{\h} v- v\|_{0,\K}+\|v-\mathcal{P}_{\K} v \|_{0,\K} \right)+ C|v|_{1,\K} &&\\
&\leq C(\h)^{-1}\| P_{h}^{\h} v- v\|_{0,\K} +C\|v\|_{1,\K}\;. &&
        \end{align*}
 The Stability now follows immediately, by summing over all elements $\K\subset \O_{i}$,  using the definition of the weighted $H^{1}$-semi-norm and the weighted $L^{2}$-norm together with the approximation result already shown:  
  \begin{align*}
     | P_{h}^{\h} v |_{1, \kappa, \O} &\leq C \h^{-1}\| P_{h}^{\h} v- v\|_{0,\kappa, \O} +\|v\|_{1,h,\kappa,\O}  &&\\
     &\leq C\h^{-1} \h \left(\log{\left(\frac{2\h}{h}\right)}\right)^{1/2} \| v\|_{1,h,\kappa, \O}  + \|v\|_{1,h,\kappa,\O}     &&\\
     &\lesssim   \left(\log{\left(\frac{2\h}{h}\right)}\right)^{1/2} \left\| v\right\|_{1,h,\kappa, \O}\;, &&
  \end{align*}   
   and the proof is complete.  
   \end{proof}


 \end{document}